\documentclass[pdflatex,sn-mathphys-num]{sn-jnl}

\usepackage{graphicx}%
\usepackage{multirow}%
\usepackage{amsmath,amssymb,amsfonts}%
\usepackage{amsthm}%
\usepackage{mathrsfs}%
\usepackage[title]{appendix}%
\usepackage{xcolor}%
\usepackage{textcomp}%
\usepackage{manyfoot}%
\usepackage{booktabs}%
\usepackage{algorithm}%
\usepackage{algorithmicx}%
\usepackage{algpseudocode}%
\usepackage{listings}%
\usepackage{float}
\usepackage{enumerate}
\usepackage{appendix}
\usepackage{bm}
\newcommand{\bvec}[1]{{#1}}
\newcommand{\highds}{\bvec{s}}
\newcommand{\highdtx}{\tilde{\bvec{x}}}

\theoremstyle{thmstyleone}%
\newtheorem{theorem}{Theorem}
\newtheorem{lemma}[theorem]{Lemma}

\theoremstyle{thmstyletwo}%
\newtheorem{example}{Example}%

\theoremstyle{thmstylethree}%

\newtheorem{assumption}[theorem]{Assumption}

\begin{document}

\title[Optimization-Free Diffusion Model - A Perturbation Theory Approach]{Optimization-Free Diffusion Model - A Perturbation Theory Approach}


\author[1]{\fnm{Yuehaw} \sur{Khoo}}\email{ykhoo@uchicago.edu}
\author[2]{\fnm{Mathias} \sur{Oster}}\email{oster@igpm.rwth-aachen.de}
\equalcont{Corresponding author}

\author[3]{\fnm{Yifan} \sur{Peng}}\email{yifanpeng@uchicago.edu}

\affil[1]{\orgdiv{CCAM and Department of Statistics}, \orgname{University of Chicago}, \orgaddress{\street{5747 S Ellis Avenue}, \city{Chicago}, \postcode{60637}, \state{IL}, \country{United States}}}

\affil[2]{\orgdiv{IGPM}, \orgname{RWTH Aachen}, \orgaddress{\street{Templergraben 55}, \city{Aachen}, \postcode{52062}, \state{NRW}, \country{Germany}}}

\affil[3]{\orgdiv{CCAM}, \orgname{University of Chicago}, \orgaddress{\street{5747 S Ellis Avenue}, \city{Chicago}, \postcode{60637}, \state{IL}, \country{United States}}}

\abstract{Diffusion models have emerged as a powerful framework in generative modeling, typically relying on optimizing neural networks to estimate the score function via forward SDE simulations. In this work, we propose an alternative method that is both optimization-free and forward SDE-free. By expanding the score function in a sparse set of eigenbasis of the backward Kolmogorov operator associated with the diffusion process, we reformulate score estimation as the solution to a linear system, avoiding iterative optimization and time-dependent sample generation. We analyze the approximation error using perturbation theory and demonstrate the effectiveness of our method on high-dimensional Boltzmann distributions and real-world datasets. \\
\textbf{MSC: 65C20,65M70,68T05,68T09,62G09}}

\keywords{Diffusion models, Optimization-free methods, Perturbation theory, Backward Kolmogorov operator eigenfunctions}



\maketitle

\section{Introduction}\label{sec1}
Generative modeling, one of most active areas in machine learning, aims to learn the underlying data distribution from samples so that new samples can be generated that closely resemble the original dataset. Among the leading approaches, diffusion models~\cite{hyvarinen2005estimation} have emerged as a powerful framework. These models rely on estimating the score function—the gradient of the log-probability density—to model and generate samples from complex distributions. More specifically, diffusion models typically learn the score function $\highds(t, \bvec{x})$ by minimizing the following cost functional:
\begin{equation}\label{eq: cost functional intro}
    \min_{\highds(t,\bvec{x})}\int_0^T \int_{\mathbb{R}^d} \frac{1}{2}\|\highds(t,\bvec{x}) - \nabla \log \rho_t(\bvec{x})\|^2  \rho_t(\bvec{x}) \mathrm{d}\bvec{x} \mathrm{d}t .
\end{equation}
Here, $T$ is a chosen time horizon and $\rho_t(\bvec{x})$ is the law of $\bvec{x}_t$ governed by the overdamped Langevin dynamics:
\begin{equation}\label{eq:overdamp langevin dynamic}
    d\bvec{x}_t = -\nabla V(\bvec{x}_t) \mathrm{d}t + \sqrt{2\beta^{-1}}\mathrm{d}\bvec{w}_t, \quad \bvec{x}_0 \sim \rho_0,
\end{equation}
where $V$ is a user-specified potential function, $\beta$ is the inverse temperature, and $w_t$ is standard Brownian motion. This dynamic gradually transforms initial samples from the data distribution $\rho_0$ toward a tractable base distribution $\rho_\infty$, which corresponds to the stationary distribution of the process:
\begin{equation}\label{eq: stationary distribution}
    \rho_\infty(\bvec{x}) = \frac{1}{Z} \exp(-\beta V(\bvec{x})),
\end{equation}
where $Z$ is the normalization constant.\par 
Once the score function is learned, new samples can be generated by simulating the reverse-time stochastic differential equation (SDE) corresponding to~\eqref{eq:overdamp langevin dynamic}, starting from samples generated from the base distribution $\rho_\infty$ (which approximates $\rho_T$ when $T$ is sufficiently large) \cite{anderson1982reverse} and then propagating backwards in time: 
\begin{align}\label{eq:reverse SDE}
        d\highdtx_{T - t} &= \left(\nabla V(\highdtx_{T - t}) +  2{\beta}^{-1}\highds(T-t,\highdtx_{T-t})\right) \mathrm{d}t + \sqrt{2{\beta}^{-1}} \mathrm{d}\bvec{w}_{T-t}, \highdtx_T \sim \rho_\infty.
\end{align}
After evolving this process from $t=T$ to $t=0$, the resulting samples $\highdtx_0$ are the newly generated data that closely resemble the original dataset.
\par

To represent the score function $\highds(t,\bvec{x})$ one usually employs neural networks, achieving state-of-the-art performance in generating high-quality images, audio, and other complex data types~\cite{song2019generative, dathathri2019plug, song2020improved, song2020score, meng2021sdedit}. In terms of theoretical guarantees, several recent works~\cite{lee2022convergence, lee2023convergence} have shown that a small $L_2$ error of the estimated score function implies a small discrepancy between the distribution generated by the reverse-time SDE~\eqref{eq:reverse SDE} and the underlying ground truth distribution. In the neural network setting, it is often difficult to rigorously justify the assumption of a small score estimation error due to the challenges associated with non-convex optimization and the lack of guarantees on global convergence. In this work, we address this issue by exploring an alternative approach aside of neural networks, aiming to understand and quantify how accurately the score function can be estimated based on the number of degrees of freedom in the perturbative regime.

\subsection{Contributions}
Here we summarize our contributions: 

\begin{enumerate}    
    \item We propose an optimization-free and forward SDE-free approach for diffusion models. Specifically, when representing the score function by eigenfunctions of the backward Kolmogorov operator associated with base distribution \eqref{eq: stationary distribution}, the score can be obtained via solving linear systems. Furthermore, the coefficient matrix in the linear systems can be computed without relying on any forward SDE simulation.

    \item We provide a theoretical analysis of the score estimation error using perturbation theory, under the assumption that the underlying density is a perturbation of the base distribution~\eqref{eq: stationary distribution}. We show that the score error can be bounded in terms of perturbation scale parameter and the number of basis functions used to represent the score.

    \item We overcome the curse-of-dimensionality by projecting the score function onto the span of a \emph{cluster basis}, which is a sparse subset of the high-dimensional eigenfunctions of the backward Kolmogorov operator. This basis captures low-order correlations between the variables. The coefficients for the score function can be obtained using fast numerical techniques for solving linear system in high-dimensional settings. As a result, the overall computation cost is reduced to $O(Nd^2 + N_td^2)$ cost where $N$ is the number of initial samples and $N_t$ is the number of time steps. We demonstrate the effectiveness of our method on several high-dimensional Boltzmann distributions and real-world image datasets, achieving favorable results.

\end{enumerate}

\subsection{Related Work}

\quad \textbf{Diffusion models overview:} Research on diffusion models has primarily centered around two dominant formulations. The first approach is based on forward SDE simulation \cite{ho2020denoising,sohl2015deep,song2020score,song2021maximum}, where a stochastic differential equation (typically an Ornstein–Uhlenbeck process) is used to gradually diffuse data into noise. The score function is then estimated by training a neural network. For example, the Denoising Diffusion Probabilistic Model (DDPM) \cite{ho2020denoising,sohl2015deep} uses a pair of Markov chains to progressively corrupt and subsequently recover the data, with the reverse process learned via deep neural networks. Score-based generative models (SGMs) \cite{song2019generative,song2020improved,vahdat2021score} apply a sequence of increasingly strong Gaussian noise to the data and learn the corresponding score functions across time. The second formulation is based on flow matching \cite{lipman2022flow,liu2022flow,liu2022rectified}, which directly learns a time-dependent vector field that couples the target and base distributions, rather than estimating score functions through a forward stochastic process. Extensions of this approach include method based on stochastic interpolants \cite{albergo2022building,albergo2023stochastic,albergo2023stochastic2}, which interpolate between distributions using latent variables. For an overview of these methods, see \cite{yang2023diffusion,chan2024tutorial,ma2025efficient}. Our work falls into the first category, where the forward SDE is modeled as an overdamped Langevin dynamics with a user-specified potential function. Importantly, the stationary distribution associated with this dynamic is not restricted to the Gaussian case; instead, we generalize it to a broader class of distributions that remain amenable to efficient sampling.

\textbf{Error analysis of score estimation:} Many recent works have provided an error analysis for score estimation in diffusion models when using neural network as the model class. In terms of approximation error, \cite{chen2023score} considers a local Taylor approximation of the score function using neural networks, focusing on data supported on low-dimensional linear subspaces. The work demonstrates that a simplified U-Net architecture with linear encoder and decoder can achieve an efficient score approximation. In a different direction, \cite{oko2023diffusion} develops the theoretical foundations for score approximation when the initial density lies in a Besov space. In terms of statistical error, \cite{block2020generative} analyzes the Rademacher complexity of neural network classes for score representation and \cite{chen2023score,oko2023diffusion} provide nonparametric statistical perspectives on this type of error. In terms of optimization error, \cite{han2024neural} studies the optimization guarantees for score estimation using two-layer neural networks. Additional related works are surveyed in the comprehensive overview \cite{chen2024overview}. In contrast to these approaches, our work is completely based on linear regression, thereby avoiding the intractable optimization error often associated with complex neural network architectures. Furthermore, we leverage perturbation theory to provide a clear and rigorous analysis of the approximation error, providing a new perspective on score estimation.

\subsection{Organization}

We conclude this section by outlining the organization of the paper. In Section~\ref{sec: main idea}, we introduce the main idea behind our method, explaining how the score function is represented using a set of eigenfunctions. Section~\ref{sec: 1-d case} provides a detailed explanation of our approach for single-variable distributions and Section~\ref{sec: high-d case} extends the discussion to the high-dimensional case, where we incorporate additional techniques to surmount the curse of dimensionality. In Section~\ref{sec: theory}, we present a theoretical analysis of the approximation error based on perturbation theory. Section~\ref{sec: numerical} contains our numerical results, which include both synthetic examples from Boltzmann distributions and real-world data such as the MNIST dataset. Finally, Section~\ref{sec: conclusion} concludes the paper.

\section{Main Idea -- Using Spectral Methods for Diffusion Models}\label{sec: main idea}

In this paper, we propose a new perspective on diffusion models that relies entirely on numerical linear algebra, rather than neural network representations, thereby yielding an optimization-free approach. Specifically, we represent the score function as an expansion in a set of basis functions $\{f_k\}_{k\in \mathcal{S}}$ where $\mathcal{S}$ is a chosen index set:
\begin{equation}\label{eq: score representation}
   s^{(i)}(t,\bvec{x}) = \sum_{k\in \mathcal{S}} c_{k}^{(i)}(t)f_{k} (\bvec{x}) -\beta \partial_{x_i} V(\bvec{x}) = (\bvec{c}^{(i)}(t))^T \bvec{f}(\bvec{x}) -\beta \partial_{x_i} V(\bvec{x}), 
\end{equation}
where $s^{(i)}(t,\bvec{x})$ is $i$-th component of the score $\highds(t,\bvec{x})$ and $c^{(i)}_k(t)$ are the corresponding time-dependent coefficients. Note that, according to \eqref{eq: stationary distribution}, the ground-truth score function in the limit $t \rightarrow \infty$ is $\highds(\infty, \bvec{x}) = \nabla \log \rho_\infty(\bvec{x}) = -\beta \nabla V(\bvec{x})$. Thus, the advantage of including the term $-\beta \nabla V(\bvec{x})$ in~\eqref{eq: score representation} is that the corresponding coefficients $c^{(i)}_k(t)$ naturally decline to zero as $t \rightarrow \infty$, simplifying the score estimation for large $t$.\par 

To estimate the score function, the objective in~\eqref{eq: cost functional intro} can be equivalently reformulated using integration by parts \cite{hyvarinen2005estimation} as: 
\begin{equation}\label{eq: cost functional intro2}
    \min_{\highds(t,\bvec{x})}\int_0^T \int_{\mathbb{R}^d}  \left(\frac{1}{2}\|\highds(t,\bvec{x})\|^2  + \mathrm{div}(\highds(t,\bvec{x})) \right)\rho_t(\bvec{x}) \mathrm{d}\bvec{x} \mathrm{d}t,
\end{equation}
which is often more convenient for optimization. Given a time grid $t_{\text{grid}} \subset [0, T]$, the minimization problem in~\eqref{eq: cost functional intro2} reduces to independently minimizing the coefficients $\bvec{c}^{(i)}(t)$ for each time $t \in t_{\text{grid}}$ and each dimension index $i$:
\begin{align}\label{eq: minimization sec2}
\min_{\bvec{c}^{(i)}(t)} & \frac{1}{2}\sum_{k,k' \in \mathcal{S}} c^{(i)}_k(t) c^{(i)}_{k'}(t) \int \rho_t (\bvec{x})f_k(\bvec{x})f_{k'}(\bvec{x})\mathrm{d}\bvec{x} \nonumber\\
&+ \sum_{k\in \mathcal{S}} c^{(i)}_k(t)\int \rho_t(\bvec{x}) \left(\partial_{x_i}f_k(\bvec{x})-\beta f_k(\bvec{x})\partial_{x_i}V(\bvec{x})\right)\mathrm{d}\bvec{x}. 
\end{align}
For simplicity, we denote by $A(t)\in \mathbb{R}^{|\mathcal{S}|\times |\mathcal{S}|}$ and $\bvec{b}^{(i)}(t)\in \mathbb{R}^{|\mathcal{S}|\times 1}$ the matrices:

\begin{align}\label{eq: A and b main}
  A_{k,k'}(t)=\int \rho_t (\bvec{x})f_k(\bvec{x})f_{k'}(\bvec{x})\mathrm{d}\bvec{x}, \quad b^{(i)}_k(t)=\int \rho_t(\bvec{x})\left(\partial_{x_i}f_k(\bvec{x})-\beta f_k(\bvec{x})\partial_{x_i}V(\bvec{x})\right)\mathrm{d}\bvec{x}.
\end{align}
The optimization problem \eqref{eq: minimization sec2} becomes 
\begin{equation}\label{eq: linear system intro}
    \min_{\bvec{c}^{(i)}(t)\in \mathbb{R}^{|\mathcal{S}|\times 1}} \frac{1}{2}{\bvec{c}^{(i)}(t)}^T A(t) \bvec{c}^{(i)}(t) + {\bvec{c}^{(i)}(t)}^T\bvec{b}^{(i)}(t)
\end{equation}
whose solution satisfies the linear system $A(t)\bvec{c}^{(i)}(t) = -\bvec{b}^{(i)}(t)$. To simplify the notation, we aggregate the solutions for all dimension indices into a single linear system:
\begin{equation}\label{eq: linear system main high-d}
    A(t)C(t)  = -B(t),
\end{equation}
where $C(t)=[\bvec{c}^{(1)}(t),\cdots, \bvec{c}^{(d)}(t)]\in \mathbb{R}^{|\mathcal{S}|\times d}$ and $B(t)=[\bvec{b}^{(1)}(t),\cdots, \bvec{b}^{(d)}(t)]\in \mathbb{R}^{|\mathcal{S}|\times d}$. By repeating this procedure for all $t \in t_{\text{grid}}$ and each dimension index $i$, we complete the computation of the score function $\highds(t, \bvec{x})$. We refer to our approach as \textit{optimization-free} because solving \eqref{eq: linear system main high-d} can be efficiently handled via standard linear system solvers, without the need for iterative optimization methods such as gradient descent.

The remaining task is to efficiently compute $A(t)$ and $\bvec{b}^{(i)}(t)$ in \eqref{eq: A and b main}. A common approach is to estimate the integrals inside $A(t)$ and $\bvec{b}^{(i)}(t)$ via Monte Carlo integration, using samples $\bvec{x}_t$ generated by the Euler-Maruyama scheme (or other numerical SDE methods) applied to~\eqref{eq:overdamp langevin dynamic}. This method incurs a dominant computational cost of $O(N N_t |\mathcal{S}|^2)$, where $N$ is the sample size and $N_t$ is the number of time steps. In contrast, we propose an alternative approach that avoids generating $\bvec{x}_t$ at each time $t$, allowing us to compute the required integrals more efficiently and significantly reduce the overall computational cost. \par

Here we provide the intuition of our approach. Note that the overdamped Langevin dynamics in \eqref{eq:overdamp langevin dynamic} is associated with a backward Kolmogorov operator
\begin{equation}\label{eq:K operator intro}
    \mathcal{L} = -\nabla V(\bvec{x})\cdot\nabla + \beta^{-1}\Delta,
\end{equation}
whose adjoint operator, known as the Fokker-Planck operator, describes the time evolution of probability density $\rho_t(\bvec{x})$:
\begin{equation}\label{eq: Fokker-planck}
   \partial_t \rho_t(\bvec{x}) = \mathcal L^*\rho_t(\bvec{x}), \quad \text{where}\quad \mathcal L^*\rho_t(\bvec{x}) = \nabla\cdot(\nabla V(\bvec{x})\rho_t(\bvec{x})) + \beta^{-1}\Delta\rho_t(\bvec{x}).
\end{equation}
Therefore, if we choose the basis functions $\{f_k\}_{k\in \mathcal{S}}$ as the eigenfunctions of the backward Kolmogorov operator $\mathcal{L}$: 
\begin{equation*}
   \mathcal{L}f_k=\lambda_k f_k,
\end{equation*}
the integration of these eigenfunctions w.r.t. the time-evolving density $\rho_t(\bvec{x})$ is straightforward:
\begin{equation}\label{eq: eigenfunction integral}
\int f_k(\bvec{x}) \rho_t(\bvec{x})\mathrm{d}\bvec{x} = \int f_k(\bvec{x}) e^{\mathcal L^* t}\rho_0(\bvec{x})\mathrm{d}\bvec{x} = \int \rho_0(\bvec{x})e^{\mathcal L t} f_k(\bvec{x}) \mathrm{d}\bvec{x} = e^{\lambda_k t}\int f_k(\bvec{x})\rho_0(\bvec{x})\mathrm{d}\bvec{x},
\end{equation}
where $\lambda_k$ is the corresponding eigenvalue. Given data samples $\{\bvec{x}^{(j)}\}_{j=1}^N$ from the initial distribution $\rho_0(\bvec{x})$, the integral $\int f_k(\bvec{x})\rho_0(\bvec{x})\,\mathrm{d}\bvec{x}$ in the right hand side of \eqref{eq: eigenfunction integral} can be approximated via Monte Carlo integration as $\frac{1}{N} \sum_{j=1}^N f_k(\bvec{x}^{(j)})$. This allows us to approximate the integrals involving $\rho_t(\bvec{x})$ through \eqref{eq: eigenfunction integral}, without simulating the forward SDE as is generally done. In order to efficiently exploit equation~\eqref{eq: eigenfunction integral} in the computation of the integrals inside $A(t)$ and $\bvec{b}^{(i)}(t)$, we project the product $f_k(\bvec{x})f_{k'}(\bvec{x})$ and derivatives $\partial_{x_i}f_k(\bvec{x}) - \beta f_k(\bvec{x})\partial_{x_i}V(\bvec{x})$ onto the span of eigenfunctions $\{f_k\}_{k\in \mathcal{S}'}$, i.e.
\begin{equation}\label{eq:linear interpolation}
  f_k(x)f_{k'}(x)= \sum_{l\in \mathcal{S}'} u^{(k,k')}_{l}f_l(x), \quad\partial_{x_i}f_k(x)=\sum_{l\in \mathcal{S}'} v^{(i,k)}_{l}f_l(x),\quad \partial_{x_i} V(x) =  \sum_{l\in \mathcal{S}'} q^{(i)}_lf_l(x).
\end{equation}
Then we are able to compute $A(t)$ and $\bvec{b}^{(i)}(t)$ following \eqref{eq: eigenfunction integral}: 
\begin{equation}\label{eq: compute A 1d}
    A_{k,k'}(t)=\sum_{l\in \mathcal{S}'} u^{(k,k')}_{l} \int e^{\mathcal{L}t}(f_l(x)) \rho_0(x)\mathrm{d}x = \sum_{l\in \mathcal{S}'} u^{(k,k')}_{l} e^{\lambda_l t}\int f_l(x)\rho_0(x)\mathrm{d}x,
\end{equation}
\begin{align}\label{eq: compute b 1d}
    b^{(i)}_k(t) 
    &= \int e^{\mathcal{L}t} \left( \sum_{l\in\mathcal{S}'} v^{(i,k)}_l f_l(x) 
        - \beta\sum_{j\in \mathcal{S}'} q^{(i)}_j f_k(x) f_j(x) \right) \rho_0(x)\mathrm{d}x \nonumber \\
    &= \int e^{\mathcal{L}t} \left( \sum_{l\in \mathcal{S}'} v^{(i,k)}_l f_l(x) 
        - \beta\sum_{j\in \mathcal{S}'} q^{(i)}_j \sum_{l\in \mathcal{S}'} u^{(k,j)}_l f_l(x) \right) \rho_0(x)\mathrm{d}x \nonumber \\
    &= \sum_{l\in \mathcal{S}'} v^{(i,k)}_l e^{\lambda_l t} \int f_l(x) \rho_0(x) \mathrm{d}x
     - \beta \sum_{l\in \mathcal{S}'}(\sum_{j\in \mathcal{S}'} q^{(i)}_j u^{(k,j)}_l) e^{\lambda_l t} \int f_l(x) \rho_0(x) \mathrm{d}x.     
\end{align}
The choice of the new set $\mathcal{S}'$ depends on the base distribution and its associated eigenfunctions, ensuring that the representation in~\eqref{eq:linear interpolation} holds with low or even no interpolation error. Importantly, Monte Carlo integration is required only once using $\rho_0(\bvec{x})$ at the initial step, and the resulting quantities can be reused across all time steps $t$. This strategy significantly reduces computational complexity compared to conventional simulation-based methods, where we describe our method as \textit{forward SDE-free}.

We summarize our optimization-free and forward SDE-free approach for learning the score function in the following Algorithm~\ref{algorithm: main}. Then later in Section~\ref{sec: 1-d case} and Section~\ref{sec: high-d case}, we elaborate on the implementational details of Algorithm~\ref{algorithm: main}.
\begin{algorithm}[H]
\caption{Optimization-free Diffusion Model}\label{algorithm: main} 
\begin{algorithmic}[1]
    \State \textbf{Input}: Initial data $\{\bvec{x}^{(j)}\}_{j=1}^N$ sampled from $\rho_0$. Time-grid $t_{\text{grid}}\in [0,T]$. 
    \State \textbf{Select base distribution}: Select a base distribution in \eqref{eq: stationary distribution}. 
    \State \textbf{Compute eigenfunction}: Compute associated eigenfunctions of \eqref{eq:K operator intro} and choose a set of eigenfunctions $\{f_k\}_{k\in \mathcal{S}}$ for score representation.  
    \State \textbf{Solve for Score}: Compute integral $\{\int f_k(x)\rho_0(x)\mathrm{d}x\}_{k\in \mathcal{S}'}$ via Monte Carlo integration $\{\frac{1}{N} \sum_{j=1}^N f_k(\bvec{x}^{(j)})\}_{k\in \mathcal{S}'}$. $\mathcal{S}'$ is the set for expansion in \eqref{eq:linear interpolation}. 
     \For {$ t  \in t_{\text{grid}}$} 
     \State Compute coefficient $A(t)$ in \eqref{eq: compute A 1d} and $B(t)$ in \eqref{eq: compute b 1d}.
     \State Solve linear system in \eqref{eq: linear system main high-d}.
    \State Compute score function $s(t,x)$ in \eqref{eq: score representation}. 
  \EndFor
  \State \textbf{Output}: Score function $\{s(t,x)\}_{t\in t_{\text{grid}}}$. 
\end{algorithmic}
\end{algorithm}

\section{Details for Implementing the Proposed Method}
One of the key tasks in Algorithm \ref{algorithm: main} is to choose the base distribution and a rich but small subset of the corresponding eigenbasis. Furthermore, efficiently solving the score is important to make the algorithm feasible for high-dimensional problem settings. In the following, we will give some examples of base distributions and the corresponding eigenbasis. Crucially, in order to make the spectral decomposition of the score computationally accessible in high dimensions, one needs to go beyond smoothness assumptions and takes other hidden structures into account. To this end, we will exploit so-called locality assumptions by expanding the score in a cluster basis. 
\subsection{Application of Main Algorithm: Single-variable Distribution}
\label{sec: 1-d case}
We will start by exemplifying the necessary steps of Algorithm \ref{algorithm: main} in the case of single-variable distributions focusing on potential base distributions, eigenfunctions and the complexity of the optimization problem. All notions introduced here are also used in the multivariate case and the single-variable setting is only to have a clean presentation of the ideas.

\subsubsection{Examples for Base Distributions}

The most commonly used base distribution in diffusion models is a \emph{Gaussian distribution}, described by the potential $V(x)=\frac{1}{2}\alpha x^2$. The associated overdamped Langevin dynamic is the Ornstein-Uhlenbeck process and the associated eigenfunctions are (properly scaled) Hermite polynomials.

Alternatively, we consider a \emph{uniform distribution} over the compact domain $[-L,L]$. It turns out that by choosing the potential $V(x)=1$ and imposing periodic boundary conditions for the diffusion process, i.e. a periodic Brownian motion, the stationary measure of the associated Langevin dynamics is the uniform distribution. The corresponding eigenfunctions are trigonometric polynomials.

Further computation details can be found in \cite{pavliotis2014stochastic}.

\subsubsection{Choosing a good Ansatz based on Eigenfunction}\label{sec: 1d eigenfunction}

After choosing the base distribution and obtaining the associated eigenfunctions, we discuss how to select an appropriate subset of eigenfunctions to achieve an accurate score representation. Recall that the expansion in~\eqref{eq: eigenfunction integral} decays exponentially in time at a rate determined by $|\lambda_k|$, where all eigenvalues of the operator $\mathcal{L}$ in~\eqref{eq:K operator intro} are non-positive. Based on this observation, we sort the eigenfunctions in order of increasing $|\lambda_k|$ and prioritize those with smaller magnitudes, as they contribute more substantially to the representation. Thus, we choose the index set $\mathcal{S}$ for eigenfunctions to include the $|\mathcal{S}|$ eigenfunctions corresponding to the eigenvalues with smallest magnitudes.

\subsubsection{Computational Complexity of Solving for Score}\label{sec: solve LS 1d}

The key tasks in algorithm \ref{algorithm: main} is to solve the linear system to recover an approximation to the score function. Building the coefficients of $A$ and $b$ for time $t=0$ is straightforward through Monte Carlo integration which requires $O(N|\mathcal{S}'|)$ operations. To obtain the integrals at later times we do the expansion in~\eqref{eq:linear interpolation} and then compute $A(t)$ and vector $b(t)$ via \eqref{eq: compute A 1d} and \eqref{eq: compute b 1d}. For the two eigenfunction choices discussed above—Hermite polynomials and Fourier basis—the expansions in~\eqref{eq:linear interpolation} can hold exactly by using $2|\mathcal{S}|$ eigenfunctions with the smallest magnitudes for the expansion ($|\mathcal{S}'|=2|\mathcal{S}|$): both the product of two eigenfunctions and the derivative of an eigenfunction can be expressed as linear combinations of the basis functions. 
Moreover, since $V'(x) = \alpha x$ in the Hermite polynomial setting and $V'(x) = 0$ in the Fourier basis setting, there is no interpolation error introduced in representing the potential gradient in either case. Thus, computing the coefficients for the expansions in~\eqref{eq:linear interpolation} involves $O(|\mathcal{S}|^3)$ operations.  \par 
For each time step $t$, evaluating $A(t)$ and $\bvec{b}(t)$ as in~\eqref{eq: compute A 1d} and~\eqref{eq: compute b 1d} also requires $O(|\mathcal{S}|^3)$ operations. Solving the resulting linear system~\eqref{eq: linear system main high-d} (Line 7 of Algorithm~\ref{algorithm: main}) using standard methods likewise costs $O(|\mathcal{S}|^3)$. Assuming a total of $N_t$ time steps, the overall cost for these computations scales as $O(N_t |\mathcal{S}|^3)$. In summary, the total computational complexity of our method is $O(N|\mathcal{S}| + N_t |\mathcal{S}|^3)$. \par 
Notably, our approach requires only a single Monte Carlo simulation in Line 4 of Algorithm, which contrasts sharply with the standard Euler-Maruyama scheme. In the latter, the integrals in $A(t)$ and $\bvec{b}(t)$ are computed by samples $x_t$ generating at each time step \( t \). This results in a total computational cost of \( O(N N_t |\mathcal{S}|^2 + N_t |\mathcal{S}|^3) \), which can be significantly higher than that of our method, particularly in large-sample settings. \par

\subsection{Application of Main Algorithm: Multi-variable Distribution}\label{sec: high-d case}

The choice of the basis set and the numerical implementation as exemplified in the previous section does not readily extend to the higher dimensional setting without accounting for hidden structures that mitigate the curse of dimensionality. Especially choosing the set $\mathcal S$, i.e. fixing our basis, is essential and needs to be done carefully in order for it to be small enough to be computationally feasible but rich enough to retain enough expressive power to actually approximate the score function. 
To that end, we will assume that the density $\rho_0$ is a perturbation of the base distribution $\rho_\infty$.

\subsubsection{Aiming for the Perturbative Regime: Using Mean-field Approximations as Base Distribution}
In this subsection, we focus on the choice of base distribution in the high-dimensional setting. In general, we want to choose a distribution that is easy to sample from, e.g. seperable densities.  One can simply pick a Gaussian distribution with diagonal covariance matrix or a uniform distribution on a hypercube. We, however, propose a novel base distribution based on a mean-field approximation of the underlying data distribution, constructed directly from empirical samples.

While the true underlying density $\rho_0$ is typically not mean-field, we can construct a mean-field approximation of the form $\rho_{\mathrm{MF}}(x) = \prod_{i=1}^d \rho_i(x_i)$. This approximation enables efficient i.i.d. sampling. To determine each one-dimensional marginal $\rho_i(x_i)$, we match its first few moments to those of underlying density $\rho_0$. Specifically, for a fixed moment order $n_{\text{m}}$, we impose the constraint $
\mu_{i,j} = \int x_i^j \rho_0(\bvec{x})\, \mathrm{d}\bvec{x}$ for $0 \leq j \leq n_{\text{m}}$ where the right-hand side is estimated via Monte Carlo integration using samples drawn from $\rho_0$. Each marginal $\rho_i$ is then obtained by solving the following entropy maximization problem:
\begin{align}\label{eq: entropy maximization}
    &\max_{\rho_i} \int \rho_i(x_i) \log \rho_i(x_i)\, \mathrm{d}x_i \nonumber \\
    \text{subject to} \quad & \int x_i^j \rho_i(x_i)\, \mathrm{d}x_i = \mu_{i,j}, \quad \text{for } 0 \leq j \leq n_{ \text{m} }.
\end{align}

In order to solve the constrained optimization problem, we define the associated Lagrangian 

\begin{equation*}
L(\rho_i,\{\nu^{(i)}_j\}_j) := \int \rho_i(x_i)\log(\rho_i(x_i))\mathrm{d}x_i + \sum_{j=0}^{n_{ \text{m} }} \nu^{(i)}_j(\int x_i^j\rho_i(x_i)\mathrm{d}x_i-\mu_{i,j}).
\end{equation*}
The dual variables $\nu_j^{(i)}$ can be computed via standard numerical optimization approach \cite{boyd2004convex,wainwright2008graphical} and $\rho_i$ are obtained through these dual variables via the first order necessary conditions 
\begin{equation*}
  \frac{\delta L(\rho_i,\{\nu^{(i)}_j\}_j)}{\delta \rho_i} = 0 \implies \rho_i(x_i) = \exp(-\sum_{j=0}^{n_{ \text{m} }} \nu^{(i)}_j x_i^j - 1). 
\end{equation*}

After getting $\{\rho_i\}^d_{i=1}$ that defines $\rho_{\mathrm{MF}}$, we let $\rho_\infty = \rho_{\mathrm{MF}}$, which serves as the high-dimensional base distribution.

\subsubsection{Going beyond Smoothness: Using Clusters of Eigenfunction}\label{sec: cluster truncation}
For our proposed method, the expansion of the score function in terms of eigenfunctions associated to the base distribution is crucial. For non-standard, seperable Gibbs measures, we need to approximate the eigenfunctions numerically as follows. Let $\rho_\infty(x) \propto \exp(- V(x))=\exp(-\sum^d_{i=1} V_i(x_i))$, where each $V_i$ is a single variable function. In this case, the high-dimensional backward Kolmogorov operator can be decomposed as the sum of one-dimensional operators

\begin{align}
    \mathcal{L} &= -\nabla V(\bvec{x})\cdot\nabla + \beta^{-1}\Delta = 
    \sum_{i=1}^d(-\frac{\partial V(\bvec{x})}{\partial x_i} \frac{\partial}{\partial x_i} + \beta^{-1}\frac{\partial^2}{\partial x_i^2}) \nonumber\\
    &  = 
    \sum_{i=1}^d(-\frac{\partial V_i(\bvec{x})}{\partial x_i} \frac{\partial}{\partial x_i} + \beta^{-1}\frac{\partial^2}{\partial x_i^2}) := \sum_{i=1}^d \mathcal{L}_i.
\end{align}

Let $\lambda^{(i)}_{n_i}$ and $\phi^{(i)}_{n_i}(x_i)$ denote the eigenvalues and eigenfunctions of the one-dimensional operator $\mathcal{L}_i$. Then, the eigenvalues and eigenfunctions of the full high-dimensional operator $\mathcal{L}$ are given by \ $\sum_{i=1}^d \lambda^{(i)}_{n_i}$ and $\prod_{i=1}^d \phi^{(i)}_{n_i}(x_i)$, respectively. These eigenfunctions can be constructed by solving individual 1D eigenvalue equation with a finite difference discretization.  

For the next step, we demonstrate how to select a subset of eigenfunctions for score representation. In the single-variable case, we select the  eigenfunctions with small eigenvalues, which typically gives rise to smooth approximations. However, one cannot directly extend this strategy of choosing eigenvalues lesser than a certain threshold, directly to the high-dimensional setting. The reason is that such a choice corresponds to selecting the interior of a ball, with a weighted norm induced by the eigenvalues, in $\mathbb Z^d$ with a volume that scales exponentially in $d$. Therefore, it is crucial to make a well-informed choice of subsets of these balls by adopting different strategies to select the eigenfunctions. We, thus, consider a structured subset of eigenfunctions known as the \emph{$j$-cluster} basis \cite{chen2023combining,peng2023generative}, which consists of all $j$-variable basis functions:

\begin{equation}\label{k-cluster basis}
\mathcal{B}_{j,d} = \left\{ \phi^{(i_1)}_{n_{i_1}}(x_{i_1}) \cdots \phi^{(i_j)}_{n_{i_j}}(x_{i_j}) \;\middle|\; 1 \leq i_1 < \cdots < i_j \leq d,\; 1 \leq n_{i_1}, \dots, n_{i_j} \leq n \right\},
\end{equation}
where $n$ represents the number of one-dimensional eigenfunctions for each dimension. Each eigenfunction in $\mathcal{B}_{j,d}$ involves at most a product of $j$ one-dimensional eigenfunctions. This approach can reduce the total number of basis functions from $n^d$ to ${d \choose j} n^j$, offering substantial computational savings while preserving expressive power.

To further reduce the computational complexity, we introduce two assumptions on the underlying density $\rho_0$: 

\begin{itemize}
    \item We assume that a 2-cluster basis is sufficient to represent the score function. This assumption seems plausible, if the target density $\rho_0$ takes a high-dimensional Gibbs measure of the form $\exp(-V(x))$, where $V(x)$ is at most a 2-cluster function. 
    \item Additionally, we assume that it suffices to use a sparse set of 2-cluster basis to represent the score function. This is possible, if the potential $V(x) = \sum_{(i,j)} V_{ij}(x_i,x_j)$ for the target Gibbs measure only couples a sparse set of $(i,j)$ pair. This type of potentials are quite common in physics application where there is near-sightedness. 
\end{itemize}
In Example 3 in Appendix~\ref{appendix: example} we demonstrate how such structures transfer from the potential to the score function in the perturbative regime. These structural assumptions allow us to expand the score function using a local 2-cluster basis, where basis functions are restricted to interactions between nearby coordinates:
\begin{equation}\label{eq: 2-cluster local basis}
\mathcal{B}^{\text{local}}_{2,d} = \left\{ \phi^{(i)}_{n_i}(x_i)\, \phi^{(i')}_{n_{i'}}(x_{i'}) \;\middle|\; 1 \leq i < i' \leq d,\; i' \leq i + d_\text{b},\; 1 \leq n_i, n_{i'} \leq n \right\},
\end{equation}
where $d_\text{b}$ is a dimension-independent bandwidth parameter that controls the locality of interactions. Under this formulation, when the score function lies within $\mathcal{B}^{\text{local}}_{2,d}$, the number of coefficients in $\bvec{c}^{(i)}(t)$ is reduced to at most $|\mathcal{S}|\leq d_\text{b} d n^2$, resulting in a representation whose complexity grows linearly with the dimension $d$.

\subsubsection{Using Thresholding and Sketching to Efficiently Solve for the Score in High Dimensions}\label{sec: fast numerical trick}
Let us now review the computational complexity of solving the linear system exploiting the local 2-cluster basis  $\mathcal{B}^{\text{local}}_{2,d}$ in \eqref{eq: 2-cluster local basis}. We summarize the key components of the method and focus on fast numerical techniques, detailed computational steps for \eqref{eq: compute A 1d} and \eqref{eq: compute b 1d} are provided in Appendix~\ref{appendix: coefficient computation high-d}.

In the initial step (Line 4 of Algorithm~\ref{algorithm: main}), we compute integrals of each pair of basis functions in $\mathcal{B}^{\text{local}}_{2,d}$ against $\rho_0(x)$ via Monte Carlo integration using its samples $\{\bvec{x}^{(j)}\}_{j=1}^N$. This step incurs a computational cost of $O(Nd_{\text{b}}^2d^2n^4)$. Next, we compute the coefficients for the expansions in~\eqref{eq:linear interpolation}, which serve as a preparatory step for efficiently evaluating $A(t)$ and $b(t)$ in~\eqref{eq: A and b main}. We use the same choice of $\mathcal{S}'$ as in Section~\ref{sec: solve LS 1d} for the expansions. This step requires $O(d_{\text{b}}dn^3)$ operations. For each time step $t$, the computation of $A(t)$ and $B(t)$ in~\eqref{eq: A and b main} costs $O(d_{\text{b}}^2d^2n^4 + d_{\text{b}}d^2n^3) = O(d_{\text{b}}^2d^2n^4)$. Solving the linear system~\eqref{eq: linear system main high-d} using a standard solver requires $O(d_{\text{b}}^3d^3n^6)$ operations. Therefore, the overall computational complexity of the method is 
$O(Nd_{\text{b}}^2d^2n^4 + N_td_{\text{b}}^3d^3n^6)$
where $N_t$ denotes the total number of time steps.

Moreover, in the high-dimensional settings, the matrix $A(t)$ in~\eqref{eq: A and b main} is often ill-conditioned, and directly inverting it can result in an unstable score estimator. To address this issue, a common and effective strategy is to apply singular value thresholding: we specify a threshold and discard all singular values of $A(t)$ that fall below it, then compute the pseudoinverse using the remaining components. An alternative approach is to incorporate a regularization term, such as $\ell_2$ regularization, when solving the linear system~\eqref{eq: linear system main high-d}. In our experiments, we observe that $A(t)$ frequently exhibits a low-rank structure, and thus we adopt the thresholding approach in our method. \par

One dominant term in the computational complexity described above $O(N_t d_{\text{b}}^3d^3n^6)$ arises from solving the linear system~\eqref{eq: linear system main high-d}. Fortunately, by exploiting the low-rank structure of $A(t)$ discussed above, this cost can be significantly reduced using randomized low-rank approximation techniques~\cite{liberty2007randomized, halko2011finding}. The key idea is that a low-rank approximation can be constructed by accurately capturing the dominant left singular subspace of the matrix, without explicitly reconstructing its full column space. Specifically, we can generate a random Gaussian matrix $\Omega \in \mathbb{R}^{|\mathcal{S}| \times \tilde{r}_A}$ with $\tilde{r}_A \ll |\mathcal{S}|$, typically independent of the problem dimension, and we then estimate its left singular subspace using the reduced matrix $A(t)\Omega$:
$$
[U(t), \Sigma(t), V(t)] = \text{SVD}(A(t)\Omega, r_A),
$$
where $\text{SVD}(A(t)\Omega, r_A)$ denotes a truncated SVD with rank $r_A$, and $U(t) \in \mathbb{R}^{|\mathcal{S}| \times r_A}$ is the resulting orthonormal basis. The rank $r_A$ is typically smaller than $\tilde{r}_A$ and can be adaptive to the problem. The matrix $U(t)U(t)^T A(t)$ then serves as a low-rank approximation to $A(t)$, computed at a cost of $O(|\mathcal{S}|^2 \tilde{r}_A)$. The corresponding linear system can be approximated by projecting onto this low-rank subspace:
\begin{equation}\label{eq: fast numerical}
    U(t)U(t)^T A(t) C(t) = -B(t) \quad \Rightarrow \quad \left(U(t)^T A(t)\right) C(t) = -U(t)^T B(t),
\end{equation}
where $U(t)^T A(t)$ is a matrix of size $r_A \times |\mathcal{S}|$, significantly smaller than the original system. Solving this reduced system requires only 
$O(|\mathcal{S}|^2 \tilde{r}_A + |\mathcal{S}|^2 r_A + |\mathcal{S}| d r_A + |\mathcal{S}| r_A^2) = O(|\mathcal{S}|^2 \tilde{r}_A) $ operations, in contrast to the classical $O(|\mathcal{S}|^3)$ cost for direct inversion.

By incorporating this fast numerical technique, the overall computational complexity is reduced to
\begin{equation}
   O(N d_{\text{b}}^2 d^2 n^4 + N_t d_{\text{b}}^2 d^2 n^6),
\end{equation}
with assuming a sketch size $\tilde{r}_A = O(n^2)$, which depends only on $n$. In contrast, for conventional Euler-Maruyama simulation, the dominant computational cost is $ O(N N_t d_{\text{b}}^2 d^2 n^4)$, which is significantly higher in large-sample regimes.

\section{Error Estimation in the Perturbative Regime}\label{sec: theory}
In this section, we show that if the underlying density is a perturbation of a mean-field Gibbs measure, the approximation error can be bounded in terms of the perturbation scale. We begin by formally stating and explaining the assumptions used in the analysis.

Let $\rho_0$ denote the underlying density from which we wish to sample and assume it takes the following form:
\begin{equation}
  \rho_0(\bvec{x}) = C_\rho\, \rho_\infty(\bvec{x}) \left(1 + \delta \sum_{i=1}^\infty p_i f_i(\bvec{x})\right),
\end{equation}
where $f_i(\bvec{x})$ are the eigenfunctions associated with the base distribution $\rho_\infty$, $p_i$ are the corresponding coefficients, and $\delta$ controls the perturbation scale. The constant $C_\rho$ ensures proper normalization. The second term in the formula represents the deviation of $\rho_0$ from $\rho_\infty$. 
In the following, we state the assumptions placed on the parameters $p_i$ and $\delta$ to facilitate our analysis.

 \begin{assumption}\label{ass:measure}
     We assume the base distribution takes the form  $\rho_\infty(\bvec{x})= \frac{1}{Z}e^{-V(\bvec{x})}$
     for some analytic, coercive potential $V$ such that the eigenvalues $\{\lambda_i\}$ of the associated operator $\mathcal{L}$ satisfy $0=\lambda_0>\lambda_1\geq\lambda_2\geq \dots $. Let $p_i$ satisfy
     \begin{equation}\label{eq: assumption pi}
         \sup_{\bvec{x}\in\mathbb R^d}| \sum_{i=1}^\infty p_i f_i(\bvec{x})|\leq 1,
     \end{equation}
     and assume that there is a non-decreasing function $r:\mathbb Z_+\to\mathbb R_+$ with $\lim_{M\to\infty} r(M)=0$ such that \begin{equation}\label{eq: assumption r}
         \int_{\mathbb R^d} \left\|\nabla\left(\sum_{i=M}^\infty p_i f_i(\bvec{x})\right)\right\|^2\rho_0(\bvec{x})\mathrm{d}\bvec{x}\leq r(M).
     \end{equation}
\end{assumption}
To justify the assumption, we provide examples related to the Ornstein-Uhlenbeck process and the periodic Brownian motion respectively, detailed in Appendix~\ref{appendix: example}.  \par

These assumptions allow us to analyze the score estimation error in the perturbative regime. Specifically, we will demonstrate how the choice of a finite-dimensional sub-eigensystem will affect the accuracy of the estimated score function. To this end, we first decompose the total error into two components: an approximation error and a statistical error. 
\begin{theorem}
    Let $F_M = \{f_1,\dots,f_M\}$ be our hypothesis class and let $\highds_M$ be minimizing 
    \begin{equation}
        \int_0^\infty \int_{\mathbb R^d}  \left(\frac{1}{2}\|\highds_{M}(t,\bvec{x})\|^2 +\mathrm{div}(\highds_{M}(t,\bvec{x}))\right)\rho_t(\bvec{x})\mathrm{d}\bvec{x}\mathrm{d}t
    \end{equation} over $F_M$. Now let $\{\bvec{x}^{(j)}\}_{j=1}^N$ be $N$ i.i.d. samples drawn from $\rho_0$. 
    Assume that for every $\mathrm{err}_{\mathrm{stat}}>0$ with high probability $\mathbb P=\mathbb P(N,\mathrm{err}_{\mathrm{stat}},F_M)$, depending on $N$, $\mathrm{err}_{\mathrm{stat}}$ and the (Rademacher) complexity of the model class $F_M$, it holds \\
    \begin{align*}
        &\left| \int_0^\infty \int_{\mathbb R^d}  \left(\frac{1}{2}\|\highds(t,\bvec{x})\|^2 +\mathrm{div}(\highds(t,\bvec{x}))\right)\rho_t(\bvec{x})\mathrm{d}\bvec{x}\mathrm{d}t\right.\\
        &\left.-\frac 1 N\sum_{j=1}^N \int_0^\infty \frac{1}{2} \|\highds(t,\bvec{x}^{(j)}(t))\|^2 + \mathrm{div}(\highds(t,\bvec{x}^{(j)}(t)))\mathrm{d}t\right|\leq \mathrm{err}_{\mathrm{stat}}
    \end{align*}
    for all $\highds\in F_M$.
    Let $\{\bvec{x}^{(j)}(t)\}_{j=1}^N$ be the trajectories of $\{\bvec{x}^{(j)}\}_{j=1}^N$ under the overdamped Langevin dynamics induced by the base distribution. Suppose $\highds(t,\bvec{x})= \nabla \log \rho_t(\bvec{x})$ is the ground truth score function and $\highds_{M,N}$ be the minimizer of 
    
\begin{equation*}
    \frac 1 N\sum_{j=1}^N \int_0^\infty \frac{1}{2} \|\highds_{M,N}(t,\bvec{x}^{(j)}(t))\|^2 + \mathrm{div}(\highds_{M,N}(t,\bvec{x}^{(j)}(t)))\mathrm{d}t.
\end{equation*}

\noindent Then we have with probability $\mathbb P(N,\mathrm{err}_{\mathrm{stat}},F_M)$ that
\begin{align*}
    &\left|\int_0^\infty \int_{\mathbb R^d}  \left(\frac{1}{2}\|\highds_{M,N}(t,\bvec{x})\|^2-\frac{1}{2}\|\highds(t,\bvec{x})\|^2 +\mathrm{div}[\highds_{M,N}(t,\bvec{x})-\highds(t,\bvec{x})]\right)\rho_t(\bvec{x})\mathrm{d}\bvec{x}\ \mathrm{d}t\right|\\
    &\leq  2\mathrm{err}_{\mathrm{stat}} + \mathrm{err}_{\mathrm{appr}},
\end{align*}
where
    
    \begin{align*}
        \mathrm{err}_{\mathrm{appr}} & =|\int_0^\infty \int_{\mathbb R^d}  \left(\frac{1}{2}\|s_{M}(t,x)\|^2-\frac{1}{2}\|s(t,x)\|^2 +\mathrm{div}(\highds_{M}(t,x)-s(t,x)\right)\rho_t(x)\mathrm{d}x\mathrm{d}t| \\
        & = \int_0^\infty \int_{\mathbb R^d}  \frac{1}{2}\|\highds_{M}(t,x) - \nabla \log \rho_t(x)\|^2 \rho_t(x) \mathrm{d}x \mathrm{d}t.
    \end{align*}
\end{theorem}

The proof follows directly from the triangle inequality and integration by parts. In this paper, our primary focus is on analyzing the approximation error introduced by considered eigenfunction subspaces. While the statistical error is typically bounded by the Monte Carlo rate with respect to the sample size $N$, a detailed investigation of this component is left for future work. \par 

The following theorem provides an approximation error analysis based on the perturbation assumption~\ref{ass:measure}.

\begin{theorem}\label{thm:General}
    Let $M\in\mathbb N$. Let $f_{1},\dots,f_M$ be the eigenfunctions of $\rho_\infty$ and Assumption \ref{ass:measure} be fulfilled. Assume that the potential $V\in F_M$.  Then the following bound holds 

    \begin{equation}
        \min_{\highds_M\in L^2([0,\infty),F_M)}\int_0^\infty \int_{\mathbb R^d} \frac{1}{2}\|\nabla\log\rho_t(x) - \highds_M(t,x)\|^2\rho_t(x) \mathrm{d}x\ \mathrm{d}t \leq \frac{\delta^2}{|\lambda_1|}(r(M+1)+(\frac\delta{1-\delta})^2r(1)).
    \end{equation}
   
\end{theorem}
\begin{proof}
Using eigenfunctions $\{f_i\}$ and Assumption \ref{ass:measure}, we have the following representation for time-dependent density $\rho_t(x)$:
\begin{align*}
   & \rho_t(x)=e^{\mathcal{L^*}t}\rho_0(x) = C_\rho e^{\mathcal{L^*}t}\left(\rho_\infty \left(1+\delta\sum_{i=1}^\infty p_i f_i(x)\right) \right) \nonumber\\
   & = C_\rho \rho_\infty \left(1+\delta\sum_{i=1}^\infty p_i e^{\lambda_i t} f_i(x)\right)= \frac{C_\rho }{Z}e^{-V(x)}\left(1+\delta\sum_{i=1}^\infty p_i e^{\lambda_i t} f_i(x)\right)
\end{align*}
Therefore, 
\begin{align*}
    \nabla \log(\rho_t(x))  =& -\nabla V(x) + \nabla  \log(1+\delta\sum_{i=1}^\infty p_ie^{\lambda_i t} f_i(x))\nonumber\\
    =& -\nabla V(x) + \frac{  \delta\sum_{i=1}^\infty p_ie^{\lambda_i t} \nabla f_i(x)}{1+\delta\sum_{i=1}^\infty p_ie^{\lambda_i t} f_i(x)}.
\end{align*}

Recall that $\frac 1{1+y} = \sum_{k=0}^\infty (-1)^k y^k$ for $|y|<1$. With Assumption \ref{ass:measure}, we get
$$
|\sum_{i=1}^\infty p_ie^{\lambda_i t} f_i(x)| \leq |\sum_{i=1}^\infty p_i f_i(x)| \leq 1
$$

Therefore, we can expand the second term: 
\begin{align*}
    \nabla\log \rho_t(x) & = -\nabla V(x)+\delta \sum_{i=1}^\infty p_ie^{\lambda_i t} \nabla f_i(x) \nonumber\\
    &\quad+\left(\delta \sum_{i=1}^\infty p_ie^{\lambda_i t} \nabla f_i(x)\right)\sum_{k=1}^{\infty} (-1)^k\left(\delta\sum_{i=1}^\infty p_ie^{\lambda_i t} f_i(x) \right)^k
\end{align*}
Thus by choosing 
\begin{equation*}
    \highds_M(t,x) = -\nabla V(x)+\delta \sum_{i=1}^M p_ie^{\lambda_i t} \nabla f_i(x) \in L^2([0,\infty),F_M), 
\end{equation*}
we have 
\begin{align*}
    &\|\nabla\log \rho_t(x) - \highds_M(t,x)\|^2 \nonumber\\
    = & \left\|\delta \sum_{i=M+1}^\infty p_ie^{\lambda_i t} \nabla f_i(x)+ \left(\delta \sum_{i=1}^\infty p_ie^{\lambda_i t} \nabla f_i(x)\right)\sum_{k=1}^{\infty} (-1)^k\left(\delta\sum_{i=1}^\infty p_ie^{\lambda_i t} f_i(x) \right)^k\right\|^2 \nonumber\\
    \leq &2\delta^2\left[\left\|\sum_{i=M+1}^\infty p_ie^{\lambda_i t} \nabla f_i(x)\right\|^2+\left\|\sum_{i=1}^\infty p_ie^{\lambda_i t} \nabla f_i(x)\sum_{k=1}^{\infty} (-1)^k\left(\delta\sum_{i=1}^\infty p_ie^{\lambda_i t} f_i(x) \right)^k\right\|^2\right] \nonumber\\
    \leq &2\delta^2\left[\left\|\sum_{i=M+1}^\infty p_ie^{\lambda_i t} \nabla f_i(x)\right\|^2+\left\|\sum_{i=1}^\infty p_ie^{\lambda_i t} \nabla f_i(x)\sum_{k=1}^{\infty} \delta ^k\right\|^2\right] \nonumber\\
        \leq &2\delta^2e^{2\lambda_1 t}\left[\left\|\sum_{i=M+1}^\infty p_i \nabla f_i(x)\right\|^2+\frac{\delta^2}{(1-\delta)^2}\left\|\sum_{i=1}^\infty p_i\nabla f_i(x)\right\|^2\right].
\end{align*}
The first inequality is from Cauchy inequality. For the last inequality, all eigenvalues $\lambda_i$ are non-positive, so $e^{\lambda_1 t}$ is the largest coefficient among $\{e^{\lambda_i t}\}_{i=1}^\infty$. 
Therefore,
\begin{align*}
    &\int_0^\infty \int \|\nabla\log \rho_t(x)  - \highds_M(t,x)\|^2\rho_t(x) \mathrm{d}x\mathrm{d}t \nonumber\\
    & \leq\frac{\delta^2}{|\lambda_1|}\left(\int \|\nabla \sum_{i=M+1}^\infty p_if_i(x)\|^2\rho_0(x) \mathrm{d}x + (\frac{\delta}{1-\delta})^2\int\|\nabla \sum_{i=1}^\infty p_i f_i(x)\|^2\rho_0(x)\mathrm{d}x\right) \nonumber\\
    & \leq \frac{\delta^2}{|\lambda_1|}(r(M+1) + (\frac{\delta}{1-\delta})^2 r(1))
\end{align*}
where we use the fact $\int_0^\infty e^{2\lambda_1 t} \mathrm{d}t = \frac{1}{-2\lambda_1}=\frac{1}{2|\lambda_1|}$. 
\end{proof}

This result shows that the approximation error is closely tied to the perturbation scale $\delta$: the error vanishes as $\delta \to 0$. However, even as the number of basis functions $M$ tends to infinity, a small residual term in the approximation error may still persist. A more refined analysis—leveraging certain algebraic structures of the basis—can strengthen this result. We will illustrate this in the context of one-dimensional periodic Brownian motion.

\begin{theorem}\label{thm:Fourier}
    Let $f_i(x) = e^{\sqrt{-1}\omega i\, x}$ and $\lambda_i$ as the corresponding eigenvalue, $i\in \mathbb{Z}$.
     Assume coefficient $|p_i|\leq \gamma^i$ for some $\gamma<1$. Then for any $\varepsilon>0$, there is $M$ and a $s_{2M}\in \mathrm{Span}\{f_{-2M},\dots,f_{2M}\}$ 
     such that
    
    \begin{equation}
        \int_0^\infty \int (\partial_x \log \rho_t(x) - s_{2M}(t,x))^2\rho_t(x) \mathrm{d}x \mathrm{d}t\leq\frac{\varepsilon}{|\lambda_1|}
    \end{equation}
\end{theorem}
This theorem shows that the approximation error can be made arbitrarily small, provided that the number of basis functions $M$ is sufficiently large. The corresponding proof is provided in the appendix~\ref{appendix: proof}. \par

\section{Numerical Experiments}\label{sec: numerical}

In this section, we evaluate our approach on high-dimensional Boltzmann distributions and the MNIST data set. We begin by summarizing the overall numerical procedure. In the first step, we estimate the score function using Algorithm~\ref{algorithm: main} over an uniform time grid $t_{\text{grid}} \subset [0, T]$, where $T$ is chosen to be sufficiently large so that the initial data are fully diffused to the base distribution \eqref{eq: stationary distribution}. In the second step, we generate samples from the base distribution using the Metropolis–Hastings algorithm. Since our choices of base distributions are mean-field, this sampling step is relatively simple. In the third step, we simulate the reverse-time SDE \eqref{eq:reverse SDE} via Euler-Maruyama scheme over $t_{\text{grid}}$, starting from the generated samples at time $t = T$ and evolving them backward to $t = 0$. The resulting samples form an empirical distribution that approximates the target density $\rho_0$. We use the discrepancy between this empirical distribution and the ground truth distribution as the primary measure of error in our evaluation.

\subsection{One-dimensional Double-well Density}
We start our numerical experiments with a one-dimensional double-well density. A double-well density refers to a probability distribution whose potential energy function exhibits two distinct minima, or "wells," separated by a barrier. This structure often arises in physical systems, such as in statistical mechanics or molecular dynamics. The corresponding density is defined as
\begin{equation*}
    \rho_0(\bvec{x}) = \frac{1}{Z}\mathrm{e}^{-\beta_{{\mathrm{DW}}}(1 - x^2)^2/4},
\end{equation*}
where $\beta_{{\mathrm{DW}}}$ is the inverse of temperature and $Z$ is the normalization constant. We set $\beta_{{\mathrm{DW}}}=2.0$ for our test. In this experiment, we consider two choices of eigenfunctions for representing the score function introduced in Section~\ref{sec: 1-d case}: Hermite polynomials and the Fourier basis. We generate 40,000 initial samples from $\rho_0$ and evaluate the accuracy of score estimation using the relative $L_2$ score error:
\begin{equation*}
    \frac{\|s(t,x) - s^*(t,x)\|_{L_2(\rho_t)}}{\|s^*(t,x)\|_{L_2(\rho_t)}},
\end{equation*}
where $s^*(t,x)$ denotes the ground truth score function, obtained by solving the one-dimensional Fokker-Planck equation \eqref{eq: Fokker-planck} with finite difference method. This error is estimated using Monte Carlo integration with 100,000 samples from $\rho_t$ via the Euler-Maruyama scheme.\par 
Figure~\ref{fig: 1d Hermite} presents the results using Hermite polynomials as the basis. We fix potential $V(x)=\frac{1}{2}x^2$ and set the total diffusion time to $T = 2.0$ and time step $\Delta t =0.002$. Note that the initial score function is given by $s(0,x) = \partial_x \log \rho_0(x) = -\beta_{\mathrm{DW}} x(x^2 - 1)$, which is a cubic polynomial. This structure makes Hermite polynomials a natural and effective choice for the score representation.

The left subfigure of Figure~\ref{fig: 1d Hermite} shows the relative score error as a function of the number of basis functions, with $\beta = 1.0$ fixed. The error remains stable as the basis size increases, indicating good approximation properties. The middle subfigure examines the effect of varying $\beta$, which controls the standard deviation of the Gaussian base distribution. The results remain consistent across different $\beta$ values, demonstrating the robustness of our approach. Finally, the right subfigure illustrates the evolution of the score function over time for number of eigenfunctions $n = 9$ and $\beta = 1.0$, showing a transition from a cubic polynomial toward a linear function, consistent with the behavior expected under Ornstein–Uhlenbeck dynamics.

\begin{figure}[H]
    \centering
     \includegraphics[width=0.32\linewidth]{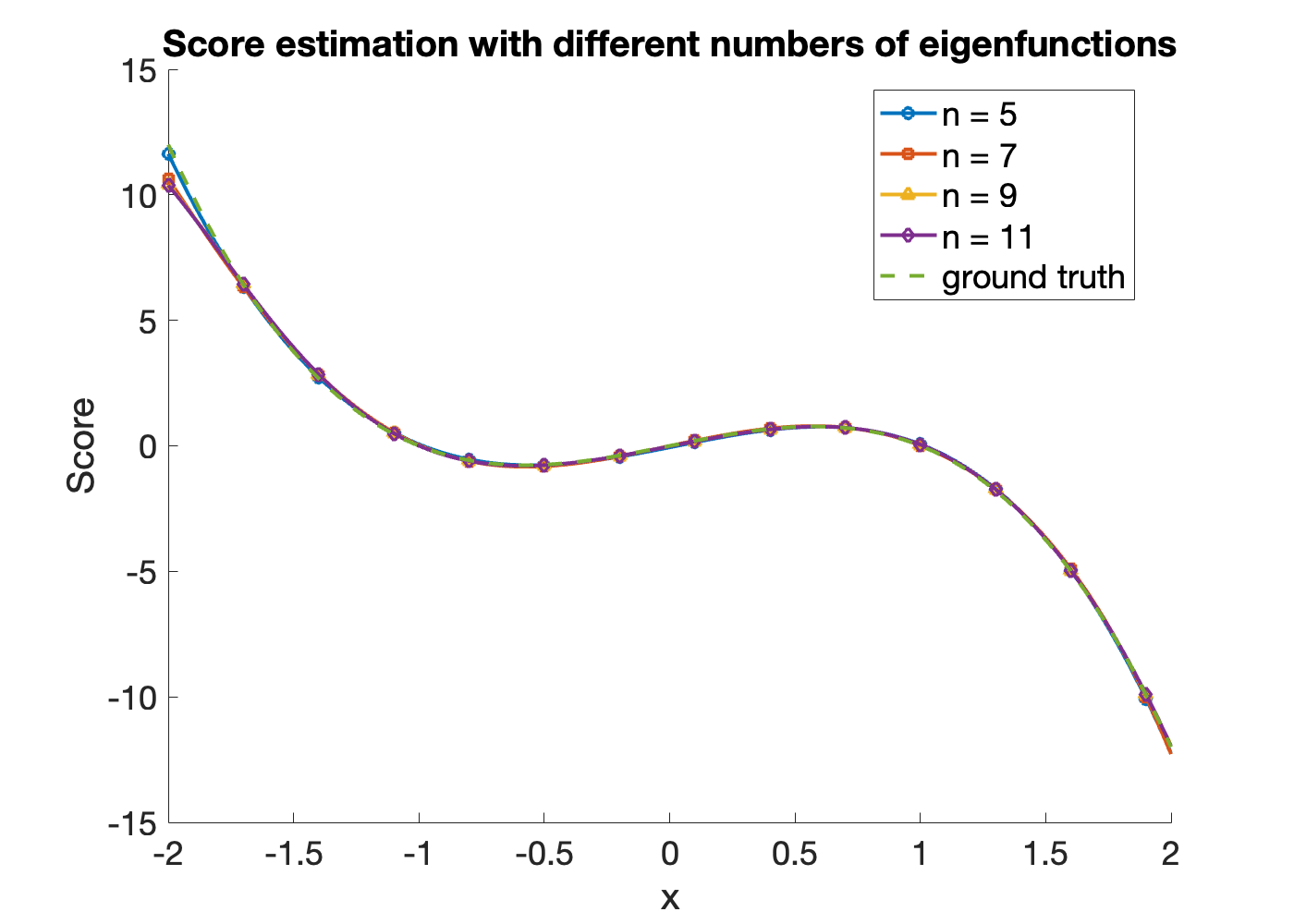}
     \includegraphics[width=0.32\linewidth]{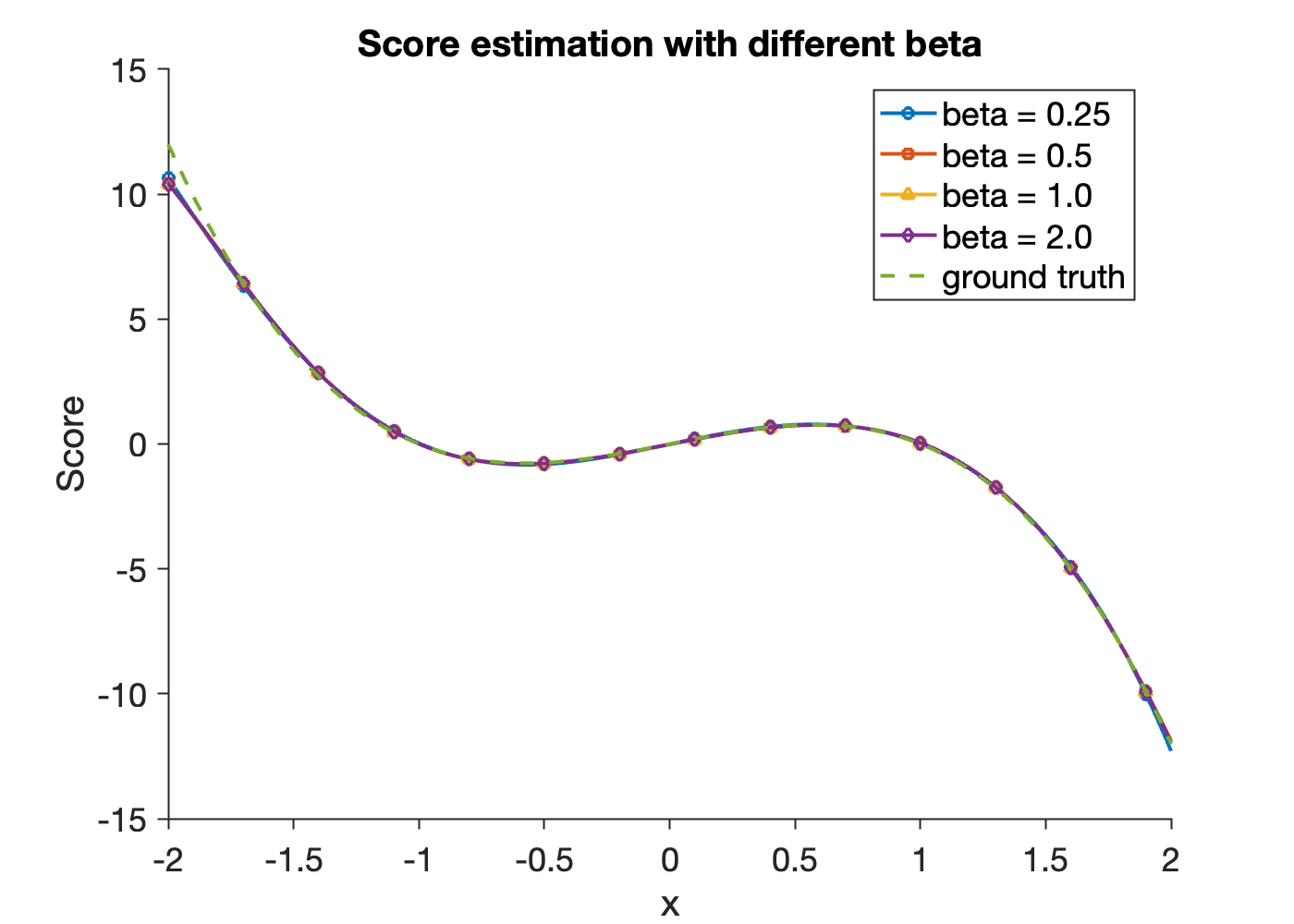}
     \includegraphics[width=0.32\linewidth]{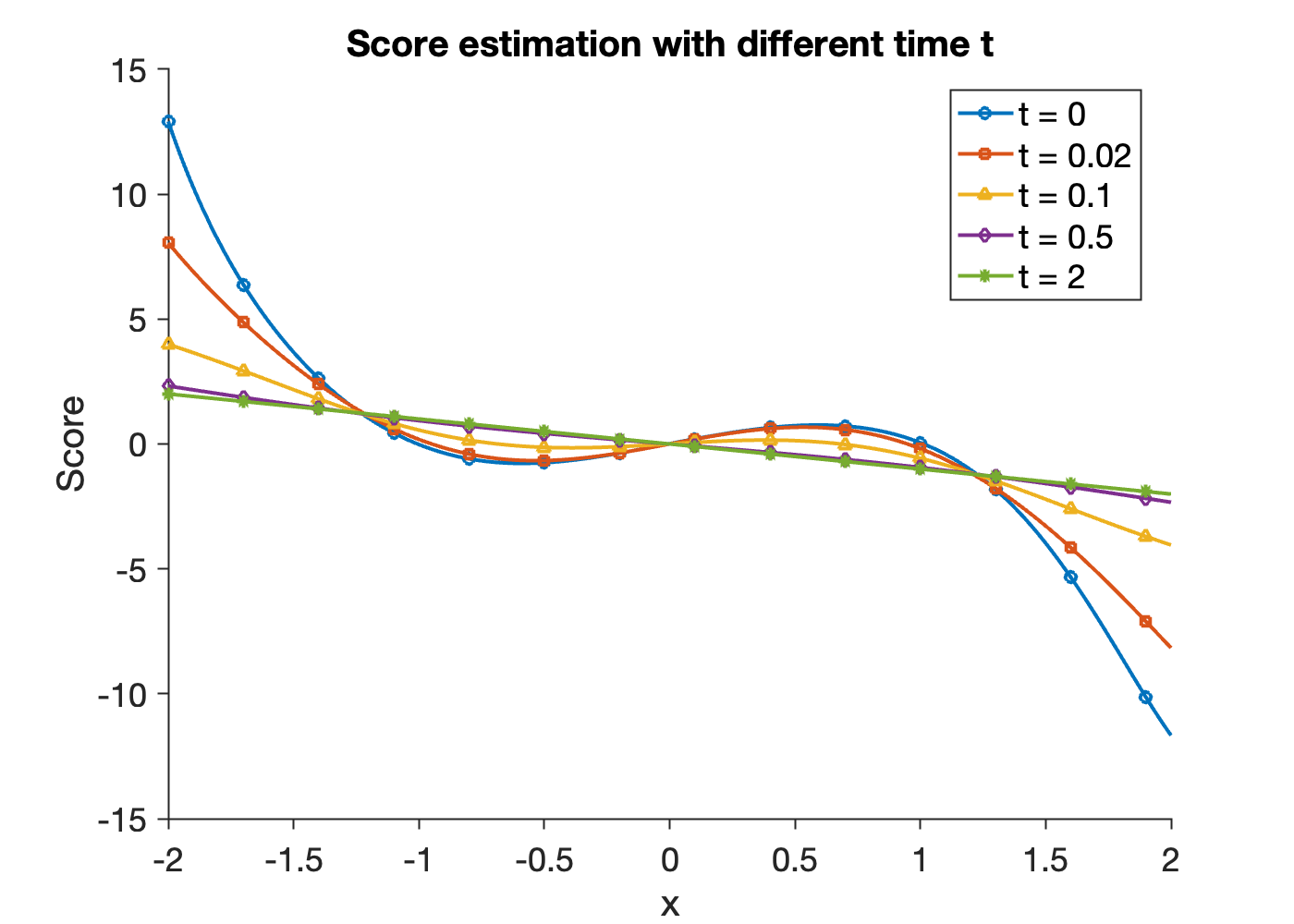}
    \caption{One-dimensional score estimation using Hermite polynomials. (Left) Estimated score functions for different numbers of eigenfunctions $n = 5, 7, 9, 11$. The corresponding relative errors are 0.0412, 0.0408, 0.0405, and 0.0401, respectively. (Middle) Score estimates for varying values of the parameter $\beta$, which controls the standard deviation of the Gaussian base distribution. The relative errors for $\beta = 0.25, 0.5, 1.0, 2.0$ are 0.0428, 0.0432, 0.0421, and 0.0421, respectively. (Right) Temporal evolution of the score function for $t = 0, 0.02, 0.1, 0.5, 2$. As expected, the score converges to a linear function over time in the Hermite polynomial representation.}\label{fig: 1d Hermite}
\end{figure}

Next, we present results using the Fourier basis, with parameters fixed as $\beta = 0.5$, $T = 2.0$ and time step $\Delta t =0.002$. The left subfigure of Figure~\ref{fig: 1d fourier} shows how the relative score error varies with the number of basis functions. As expected, the error is significantly higher when the number of basis functions is insufficient to capture the complexity of the score function. The middle subfigure illustrates the relative score error as a function of the domain size $L$. While the Fourier basis is periodic over the interval $[-L, L]$, the ground truth score function, defined on the support of the initial data distribution, is not inherently periodic. To ensure that the fitted score function accurately approximates the ground truth within the relevant region, the domain $[-L, L]$ must be large enough to cover the support of the initial distribution. This explains the poor performance when $L = 2$, as the domain is too small to capture the full support. The right subfigure visualizes the evolution of the score function over time. As time increases, the score function gradually converges to a constant, consistent with the expected long-term behavior under the chosen dynamics.

\begin{figure}[H]
    \centering
     \includegraphics[width=0.32\linewidth]{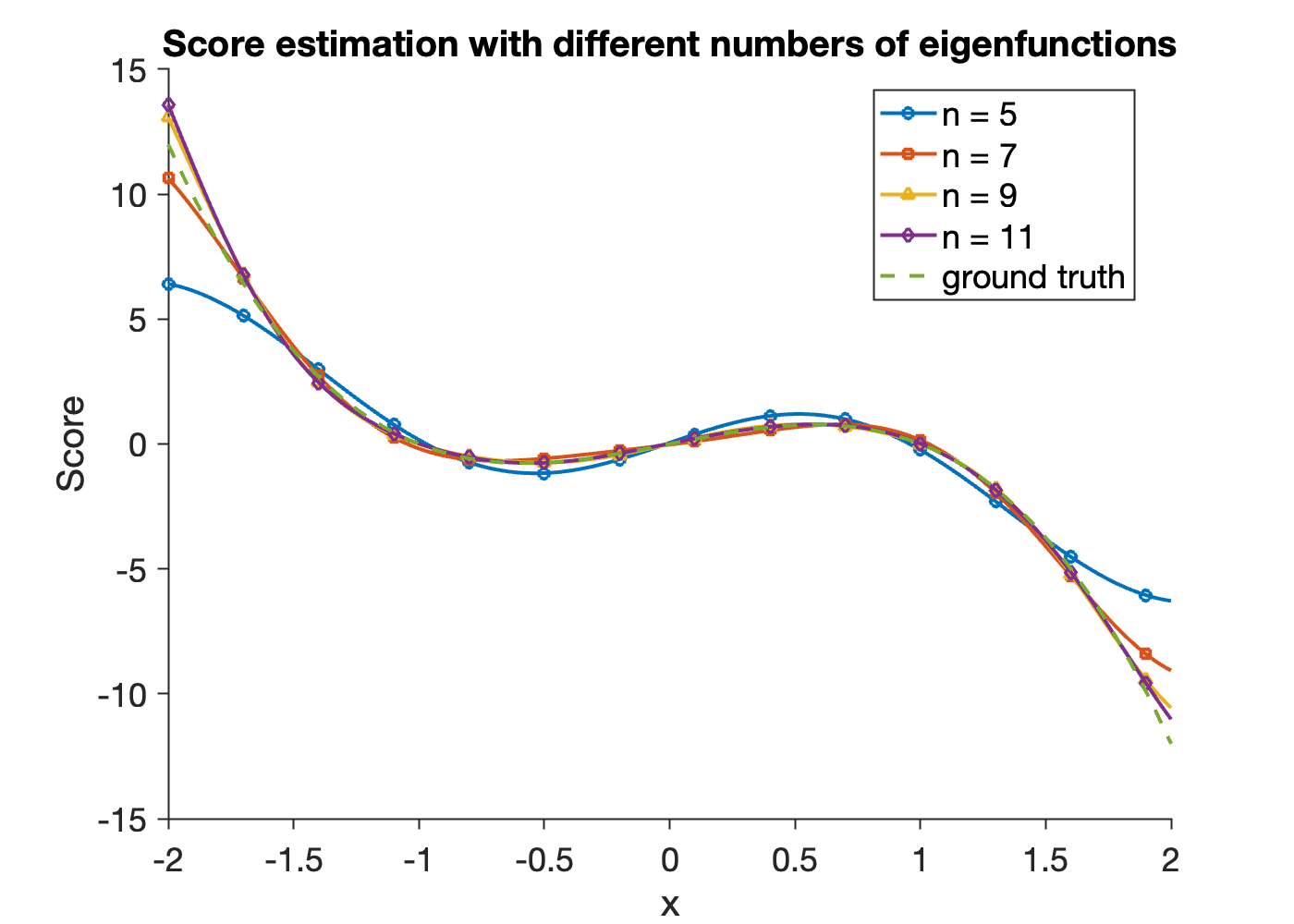}
     \includegraphics[width=0.32\linewidth]{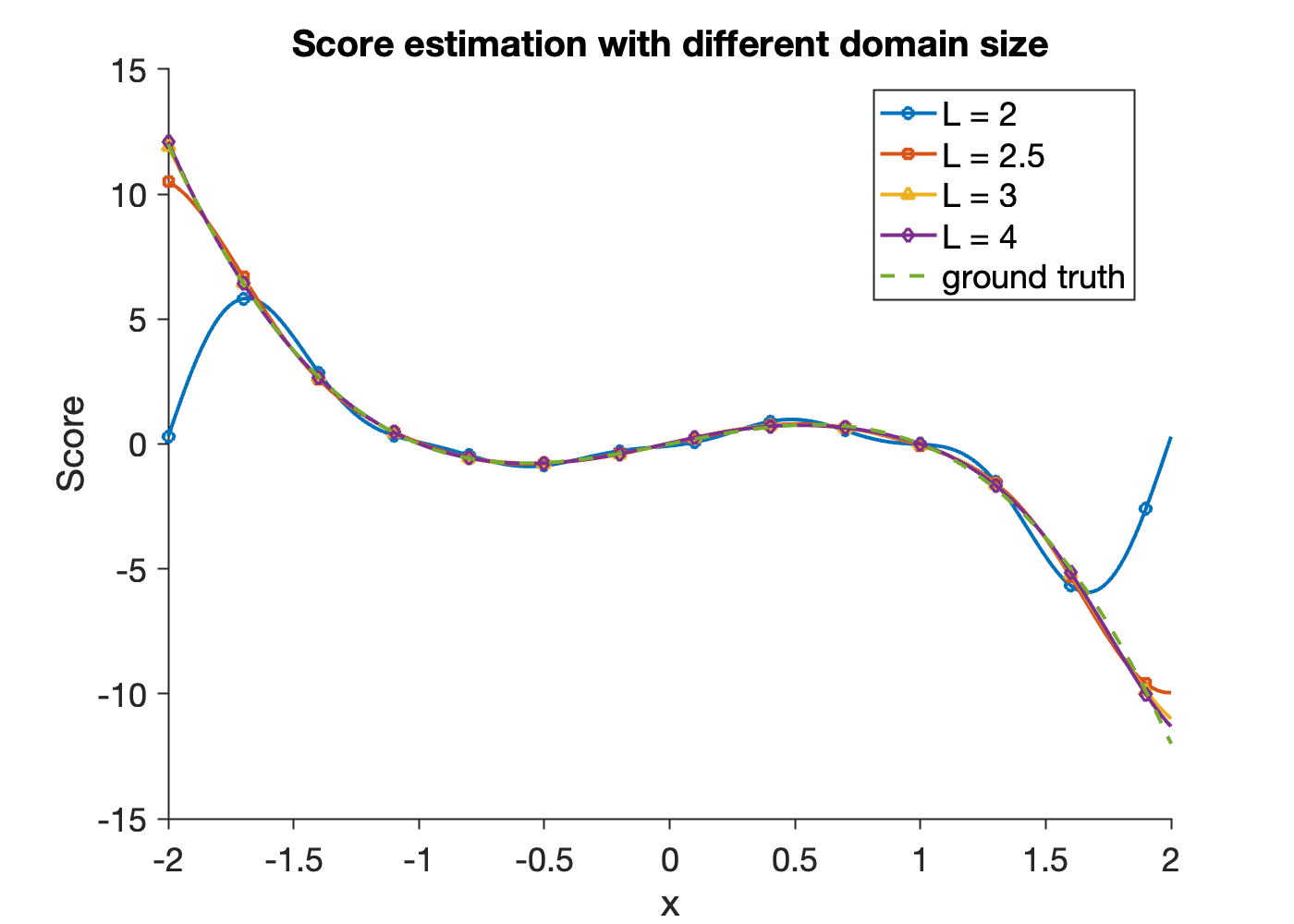}
     \includegraphics[width=0.32\linewidth]{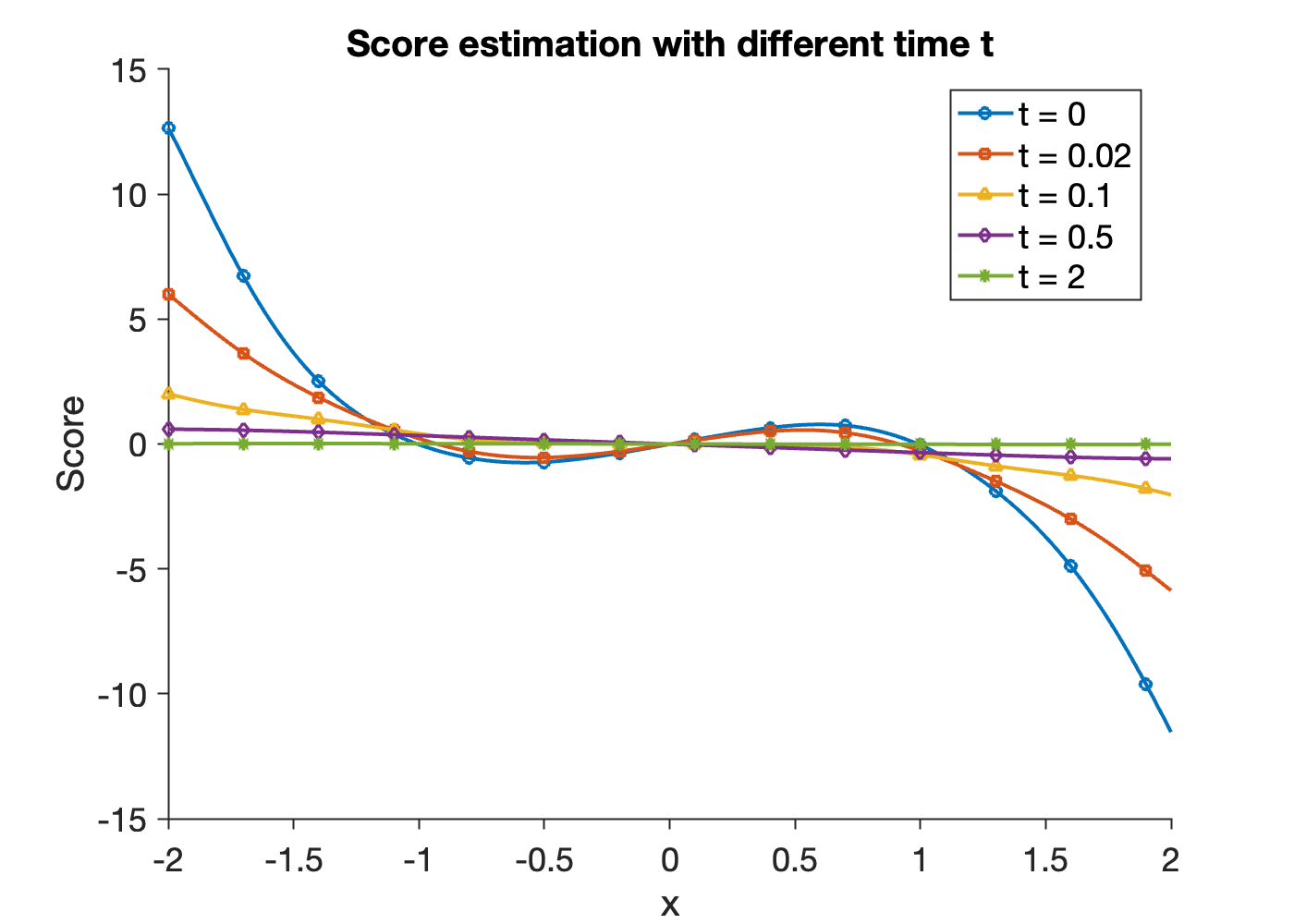}
    \caption{One-dimensional score estimation using the Fourier basis. (Left) We vary the number of eigenfunctions and visualize the estimated score. The relative score errors are 0.2756, 0.1053, 0.0666, and 0.0502 for $n = 5, 7, 9, 11$, respectively. (Middle) We vary the domain size in the Fourier basis and plot the resulting score functions. The corresponding relative score errors are 0.3612, 0.0763, 0.0501, and 0.0414 for $L = 2, 2.5, 3, 4$, respectively. (Right) We show the evolution of the score function over time for $t = 0, 0.02, 0.1, 0.5, 2$. As expected, the score function in the Fourier representation converges toward a constant zero function as time increases. }\label{fig: 1d fourier}
\end{figure}

\subsection{High-dimensional Double-well Density}
In this subsection, we consider the high-dimensional double-well density
\begin{equation*}
    \rho_0(\bvec{x}) = \frac{1}{Z}\mathrm{e}^{-\beta_{{\mathrm{DW}}} V_{{\mathrm{DW}}}(\bvec{x})}, V_{{\mathrm{DW}}}(\bvec{x})=\sum_{i=1}^d \frac{1}{4}(1 - x_i^2)^2,
\end{equation*}
where $\beta_{{\mathrm{DW}}}$ is the inverse of temperature and $Z$ is the normalization constant. The density is one type of mean-field density since the variables do not interact with each other, providing a simple high-dimensional example to illustrate the performance of our approach. We test $d=32$ and $\beta_{{\mathrm{DW}}}=8$. In this setting, we use $n = 10$ eigenfunctions per dimension and set the bandwidth to $d_\text{b} = 2$ for the local cluster basis representation~\eqref{eq: 2-cluster local basis}. The score function is estimated over the interval $t \in [0, T]$ with $T = 2$ and time-step $\Delta t = 0.002$, using 40,000 initial samples. \par 
We evaluate our method using all three types of basis functions. For the Hermite polynomial basis, we fix $\beta = 1.0$ and $V(x)=\frac{1}{2}\|x\|_2^2$ in \eqref{eq:overdamp langevin dynamic}. For the Fourier basis, we choose $L = 5.0$ and $\beta = 0.25$ in \eqref{eq:overdamp langevin dynamic}, since the dynamics converge slowly in this setting, and a smaller $\beta$ can accelerate convergence. For the mean-field approximated basis, we fix $\beta = 1.0$ and construct the approximating density by matching the first 6 moments in the entropy maximization problem~\eqref{eq: entropy maximization}.\par 
After learning the score function, we generate 40,000 samples from the base distribution and simulate the reverse-time SDE with time-step $\Delta t = 0.002$. For visualization, we primarily present one-dimensional marginal plots comparing the generated samples with reference samples. Figure~\ref{fig: 32d DW} shows the estimated density function at $x_{30}$, obtained via kernel density estimation from the generated samples. \par 

To quantify performance, we compute the relative error between the one-marginal estimated density of the generated samples and that of the reference samples, averaged across all dimensions. It is worth noting that the target density presents a significant contrast between peaks and wells, making accurate recovery particularly challenging. Despite this, the results show that the mean-field approximated basis achieves a relative error of 0.0322, demonstrating the effectiveness of our proposed approach. \par 

\begin{figure}[H]
    \centering
     \includegraphics[width=0.32\linewidth]{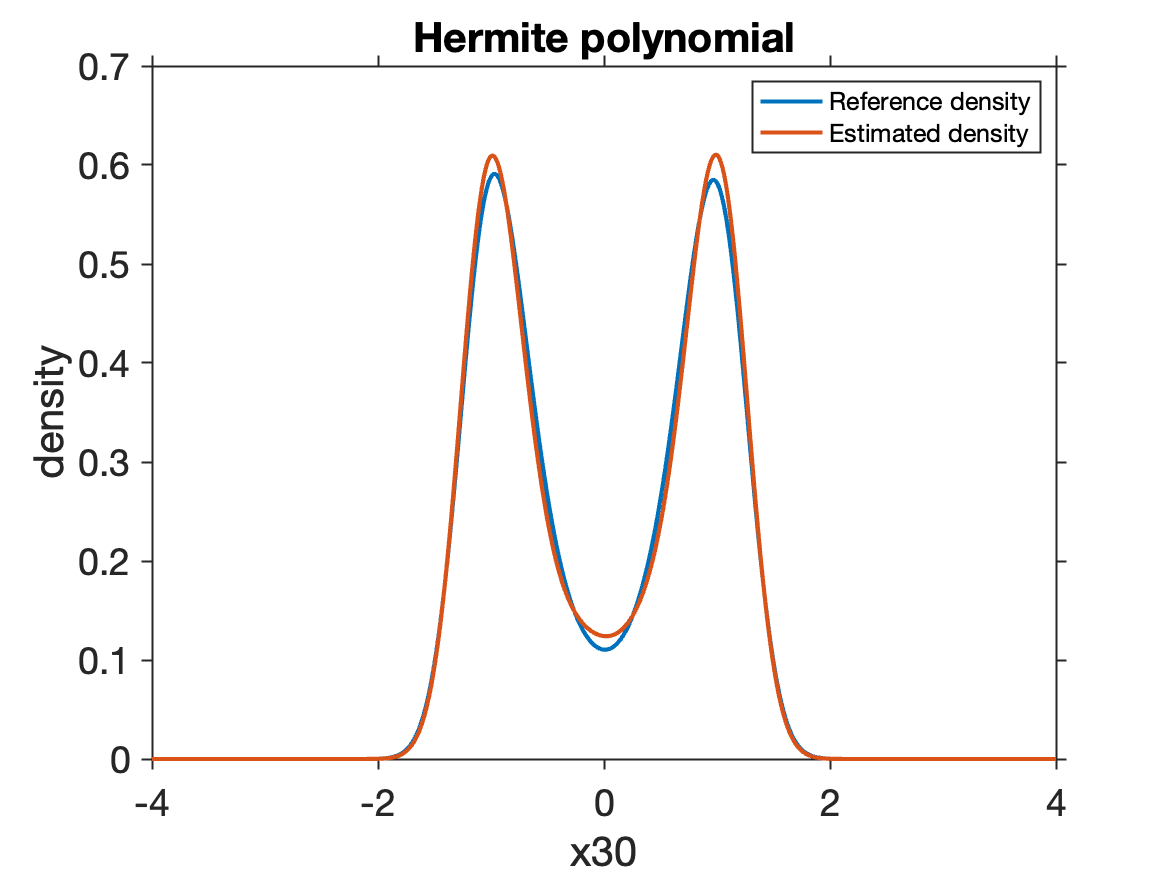}
     \includegraphics[width=0.32\linewidth]{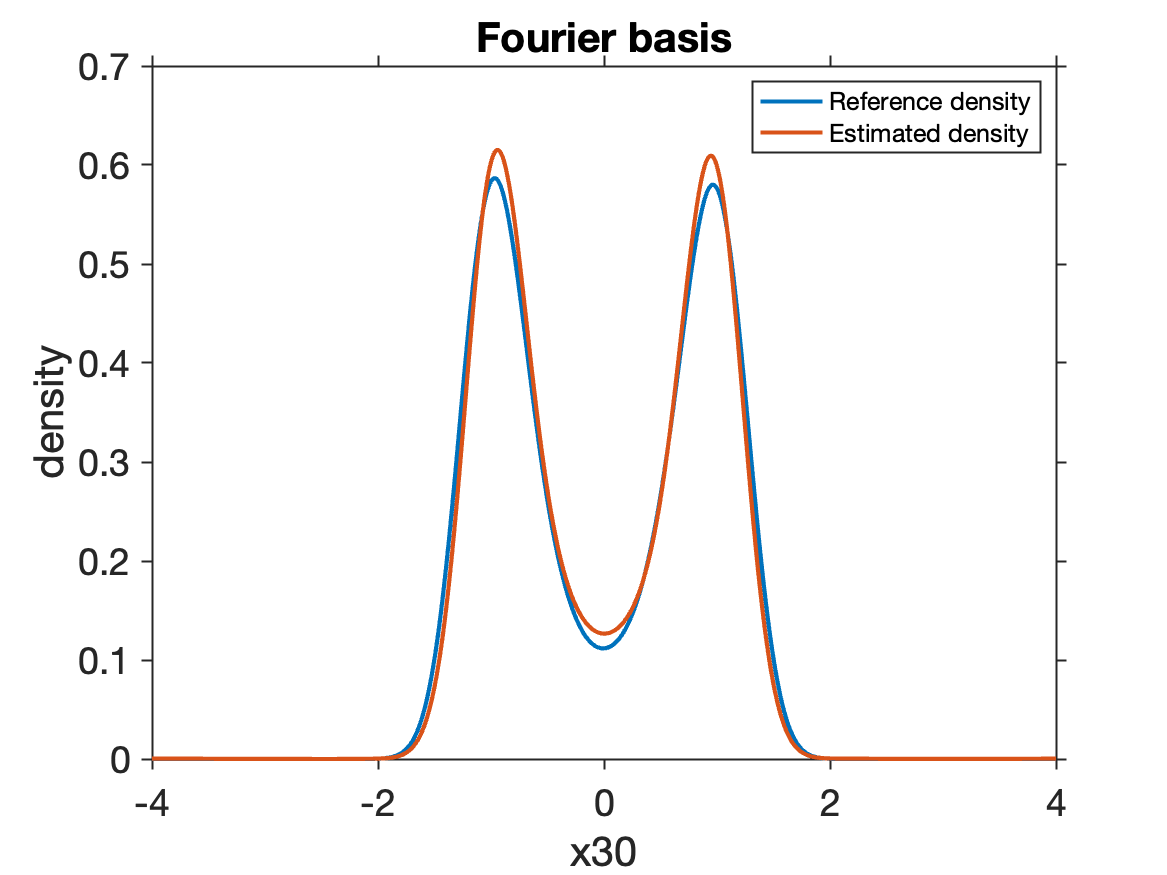}
     \includegraphics[width=0.32\linewidth]{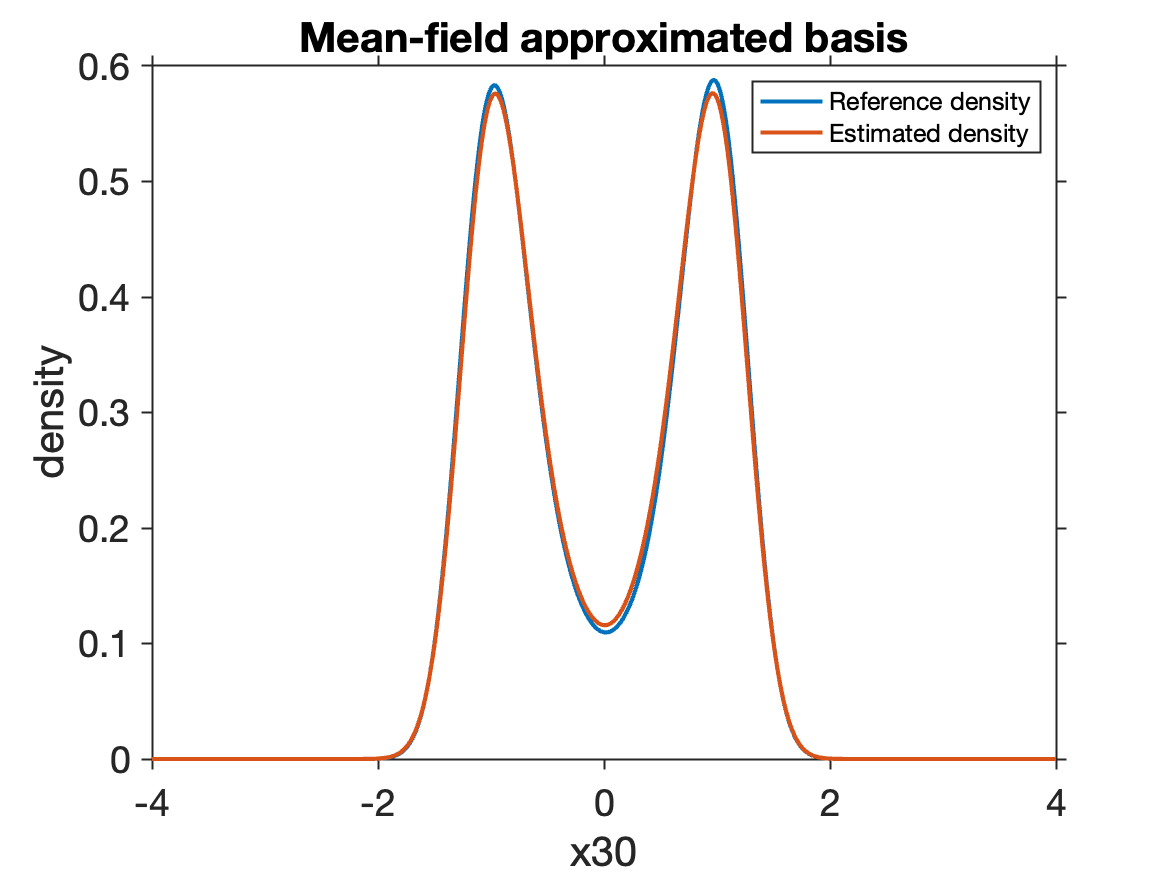}
    \caption{32-dimensional double-well density result. We visualize one-marginal densities at $x_{30}$ for three choices of eigenfunctions (Left) Hermite polynomial, (middle) Fourier basis, and (Right) Mean-field approximated basis. The relative one-marginal density errors are 0.0472, 0.0665, 0.0322, respectively.  }\label{fig: 32d DW}
\end{figure}

\subsection{High-dimensional Ginzburg-Landau Density}\label{sec: GL high-d}
In this subsection, we consider a further complicated density model, called
Ginzburg-Landau model: 
\begin{equation*}
\rho_0=\frac{1}{Z}\exp(-\beta_{\mathrm{GL}} V_{\mathrm{GL}}(x)), 
    V_{\mathrm{GL}}(\bvec{x}) = \sum^{d}_{i=1}\left( \frac{\lambda_{\mathrm{GL}}}{2} \left( \frac{x_i - x_{i-1}}{h} \right)^2 + \frac{1}{4\lambda_{\mathrm{GL}}}(1 - x_i^2)^2 \right),
\end{equation*}
for $\bvec{x} \in [-L,L]^d$, where $x_0=x_d$. In addition, the parameter $\beta_{\mathrm{GL}}$ denotes the inverse temperature, while $\lambda_{\mathrm{GL}}$ and $h$ control the interaction strength between different dimensional coordinates. In the numerical experiments, we fix $\beta_{\mathrm{GL}} = 1/8$ and $h = 0.1$, and consider various choices of $\lambda_{\mathrm{GL}}$ to evaluate the performance of our method under different interaction strength settings.\par 

For all experiments in this subsection, we use $n = 10$ basis functions per dimension. A total of 80,000 samples are drawn from the initial distribution to estimate the score function over the interval $t \in [0, T]$ with $T = 2.0$ and $\Delta t = 0.002$. For the score representation using the local cluster structure~\eqref{eq: 2-cluster local basis}, we set the bandwidth size to $d_{\text{b}} = 4$. We evaluate all three choices of eigenfunctions in our experiments. For the Hermite polynomial basis, we fix $\beta = 1.0$ and $V(x)=\frac{1}{2}\|x\|_2^2$. For the Fourier basis, we set $L = 5$ and $\beta = 0.25$. For the mean-field approximated basis, we fix $\beta = 1.0$ and construct the approximated density by matching the first 6 moments in the entropy maximization procedure. In addition, for the fast numerical technique described in~\eqref{eq: fast numerical} (Section~\ref{sec: fast numerical trick}), we select the rank that minimizes the relative error in the initial score estimation and use this same rank for score computation at all subsequent time steps $t$. \par 
In the reverse-time SDE step, we generate 80,000 samples from the base distribution and evolve them according to the learned dynamics, with time-step $\Delta t = 0.002$. To visualize the structure of the non-mean-field target density, we focus primarily on two-dimensional marginal plots, which highlight the accuracy of interactions between pairs of dimensional coordinates. To quantitatively assess the global accuracy of the generated samples, we compute the second-moment error:
\begin{equation*}
    \frac{\|\tilde{X}^T \tilde{X} - {X^*}^T X^*\|_F}{\|{X^*}^T X^*\|_F},
\end{equation*}
where $\tilde{X} \in \mathbb{R}^{N \times d}$ is the matrix of generated samples, and $X^* \in \mathbb{R}^{N \times d}$ is the matrix of reference samples.

\subsubsection{32d Result}
In this subsection, we present results for the 32-dimensional case, focusing on two choices of the interaction strength parameter. The weak interaction setting uses $\lambda_{\mathrm{GL}} = 0.03$, while the strong interaction setting uses $\lambda_{\mathrm{GL}} = 0.05$. These two choices already exhibit distinct patterns in the density 2-marginal visualizations, as illustrated in Figure~\ref{fig: 32dGL}.
The relative second-moment errors are reported in Table~\ref{tab: 32dGLresult}, which shows that the mean-field approximated basis achieves the best performance among the tested methods. \par 

Furthermore, Figure~\ref{fig: 32dGL} visualizes the two-dimensional marginal densities obtained using the mean-field approximated basis. The small discrepancy between the generated and reference distributions illustrates the effectiveness of our approach in capturing complex interactions.

\begin{table}[h]
    \centering
    \begin{tabular}{|c|c|c|c|c|c|c|}
    \hline 
         &  Mean-field approximated basis & Fourier basis & Hermite polynomial\\
         \hline 
    strong interaction   & 0.0512 & 0.0575 & 0.0588\\ 
    \hline 
    weak interaction &  0.0657 & 0.0652 & 0.0685\\
    \hline 
    \end{tabular}
    \caption{32-dimensional Ginzburg-Landau density result. We include the relative 2nd-moment error for both strong interaction case and weak interaction case among three choices of eigenfunctions. }
    \label{tab: 32dGLresult}
\end{table}

\begin{figure}[H]
    \centering
     \includegraphics[width=0.48\linewidth]{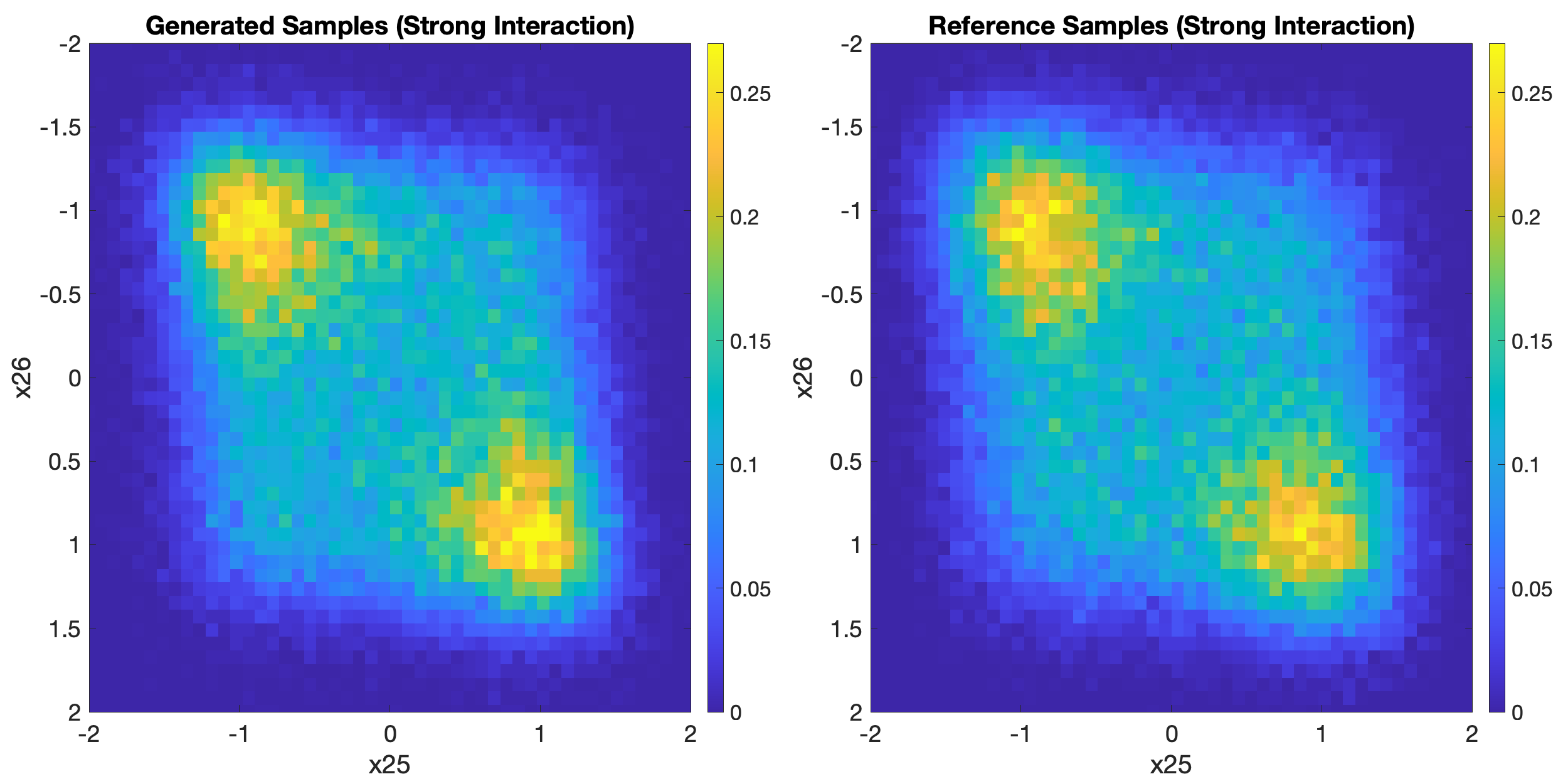}
    \includegraphics[width=0.48\linewidth]{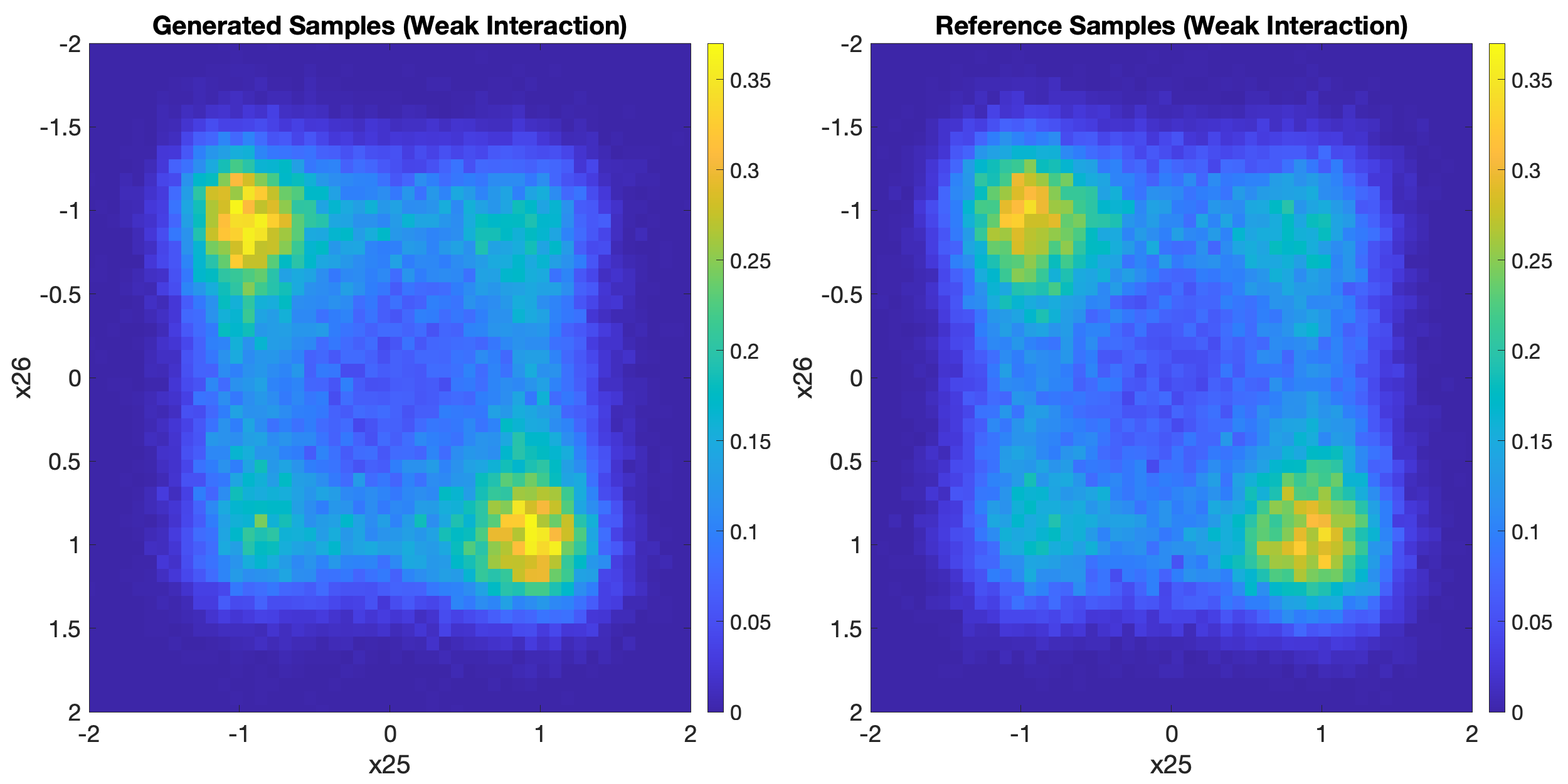}
     
    \caption{32-dimensional Ginzburg-Landau density result with mean-field approximated basis. We visualize two-marginal densities at $(x_{25},x_{26})$ for (Left) strong interaction setting and (right) weak interaction setting. }\label{fig: 32dGL}
\end{figure}

\subsubsection{64d Result}
We also evaluate our method in a 64-dimensional setting, keeping the interaction parameter fixed at $\lambda_{\mathrm{GL}} = 0.05$. In this experiment, we primarily use the mean-field approximated basis, while keeping all other parameters identical to those used in the 32-dimensional case. 

Figure~\ref{fig: 64dGL} presents selected two-marginal plots comparing the generated samples with the ground truth. The close agreement in second-moment statistics further supports the accuracy and robustness of our approach in high-dimensional settings.

\begin{figure}[H]
    \centering
     \includegraphics[width=0.48\linewidth]{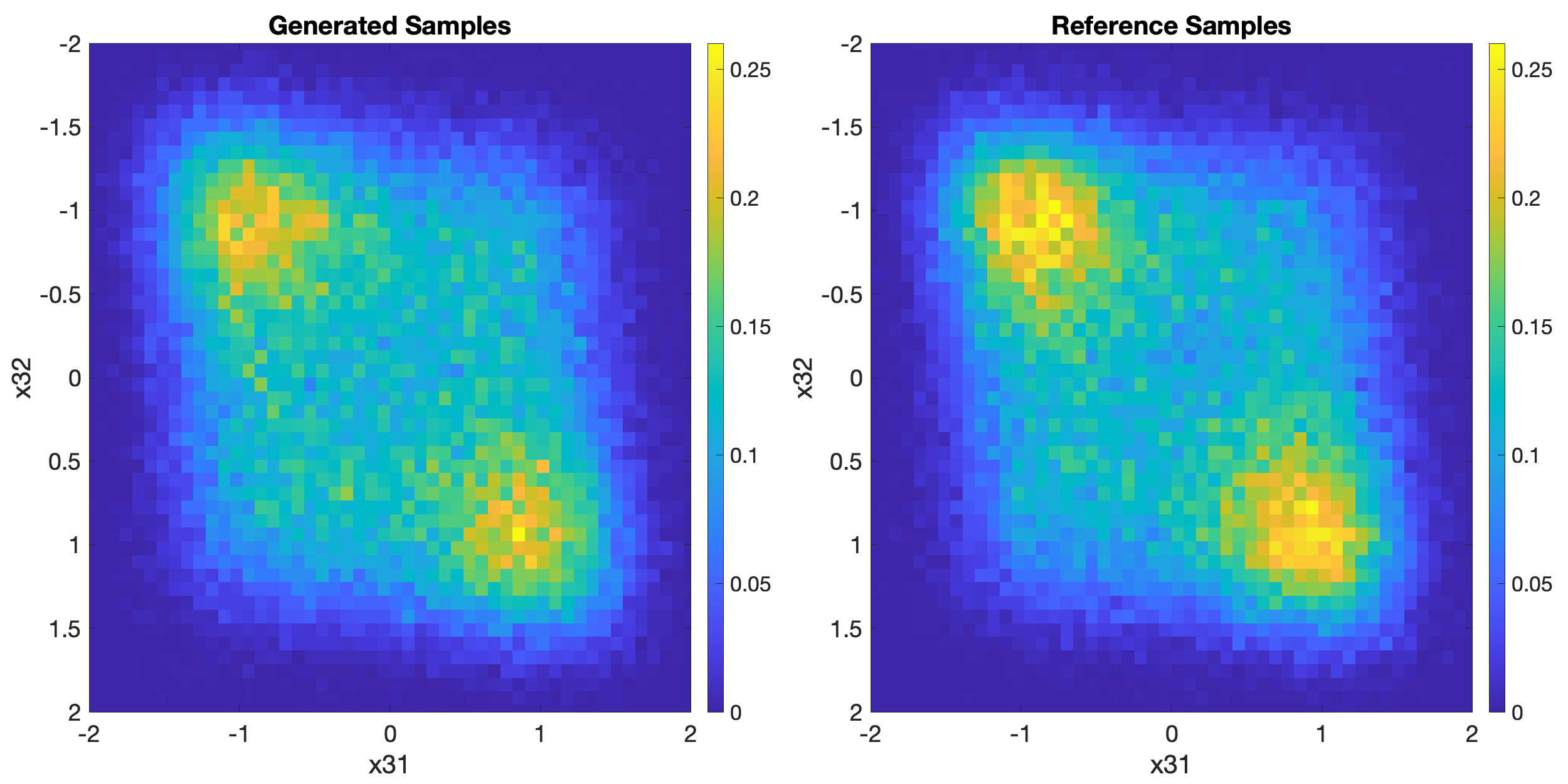}
    \includegraphics[width=0.48\linewidth]{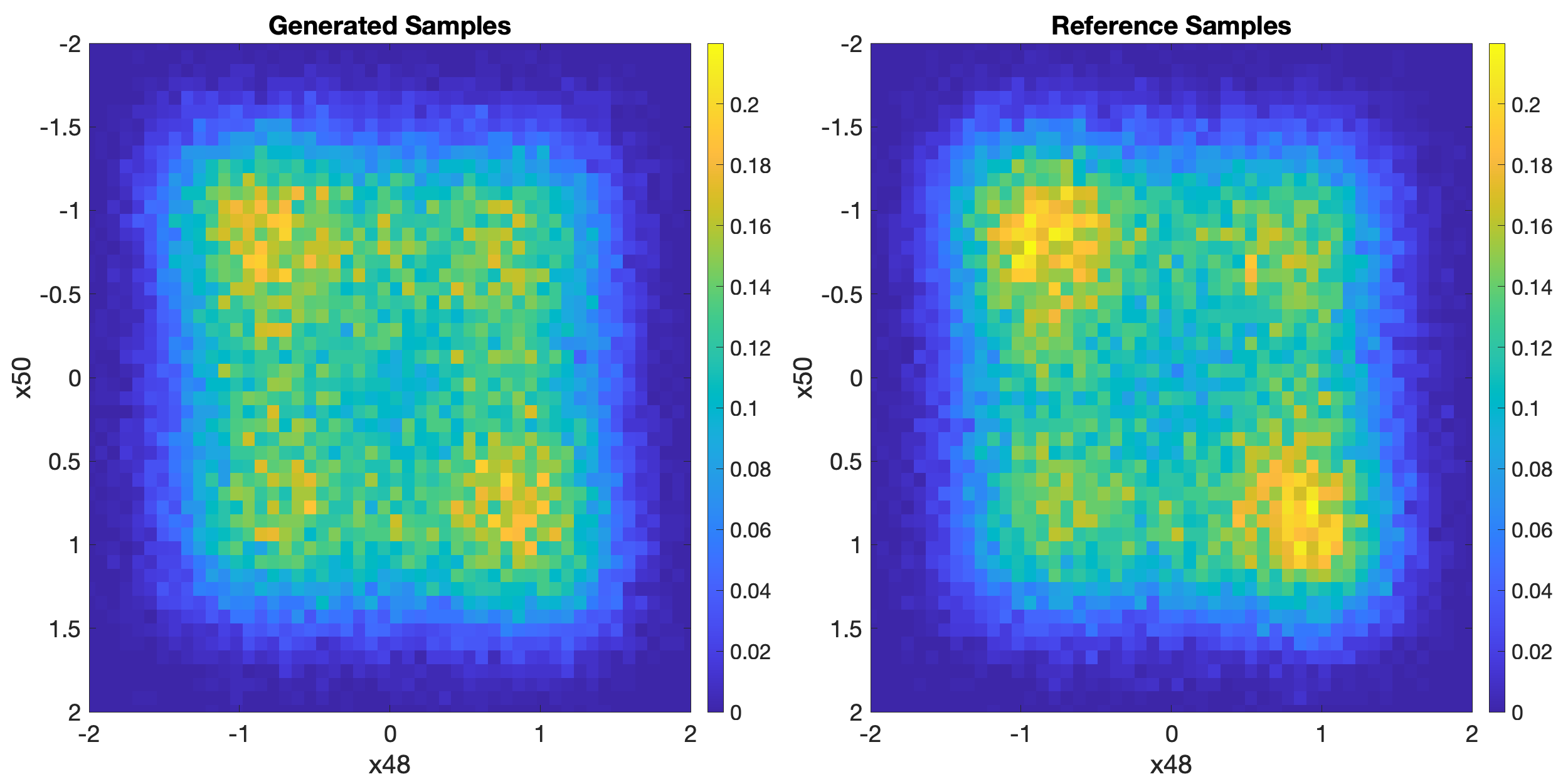}   
    \caption{64-dimensional Ginzburg-Landau density result with mean-field approximated basis. We visualize two-marginal densities (left) $(x_{31},x_{32})$  and (right) $(x_{48},x_{50})$. The relative 2nd moment error is 0.0745. }\label{fig: 64dGL}
\end{figure}

\subsection{MNIST Dataset}
In the final example, we apply our algorithm to the MNIST dataset, which consists of 70,000 images of handwritten digits (0 through 9) where each image is a $28 \times 28$ pixel matrix. In our approach, for each digit, we use 5,000 samples as training data to estimate the score function. Note that the dimensionality of each sample is $d = 28^2 = 784$, which is relatively high. To reduce the computational complexity, we first perform PCA on the training data and retain the first 10 principal components. We then apply our approach in this reduced space to generate new samples in the principal component domain. Finally, we map the generated principal components back to the original image space to recover the digit images.\par 

In the numerical experiments, we use 10 Fourier basis functions per dimension to learn the score function over the interval $t \in [0, 3]$ with time step $\Delta t = 0.002$, with a temperature parameter $\beta = 0.5$ and domain size $L = 4$. The rank used in the fast low-rank approximation technique follows the same selection strategy as described in Section~\ref{sec: GL high-d}. Figure~\ref{fig: MNIST} shows both the generated and real MNIST images. The results clearly demonstrate that the images generated by our method successfully preserve the structural features of the digits. \par

\begin{figure}[h]
    \centering
     \includegraphics[width=0.48\linewidth]{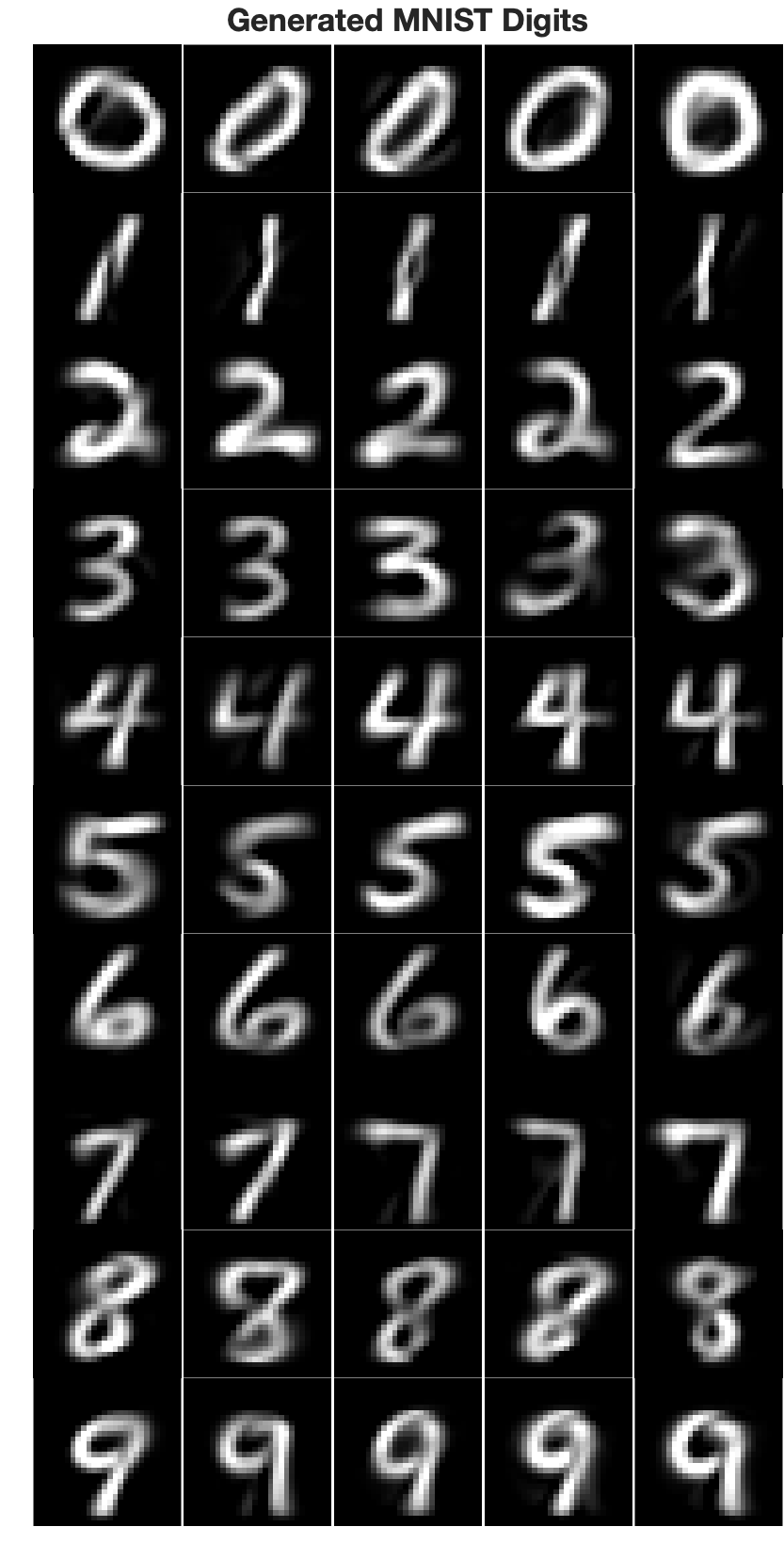} 
     \includegraphics[width=0.48\linewidth]{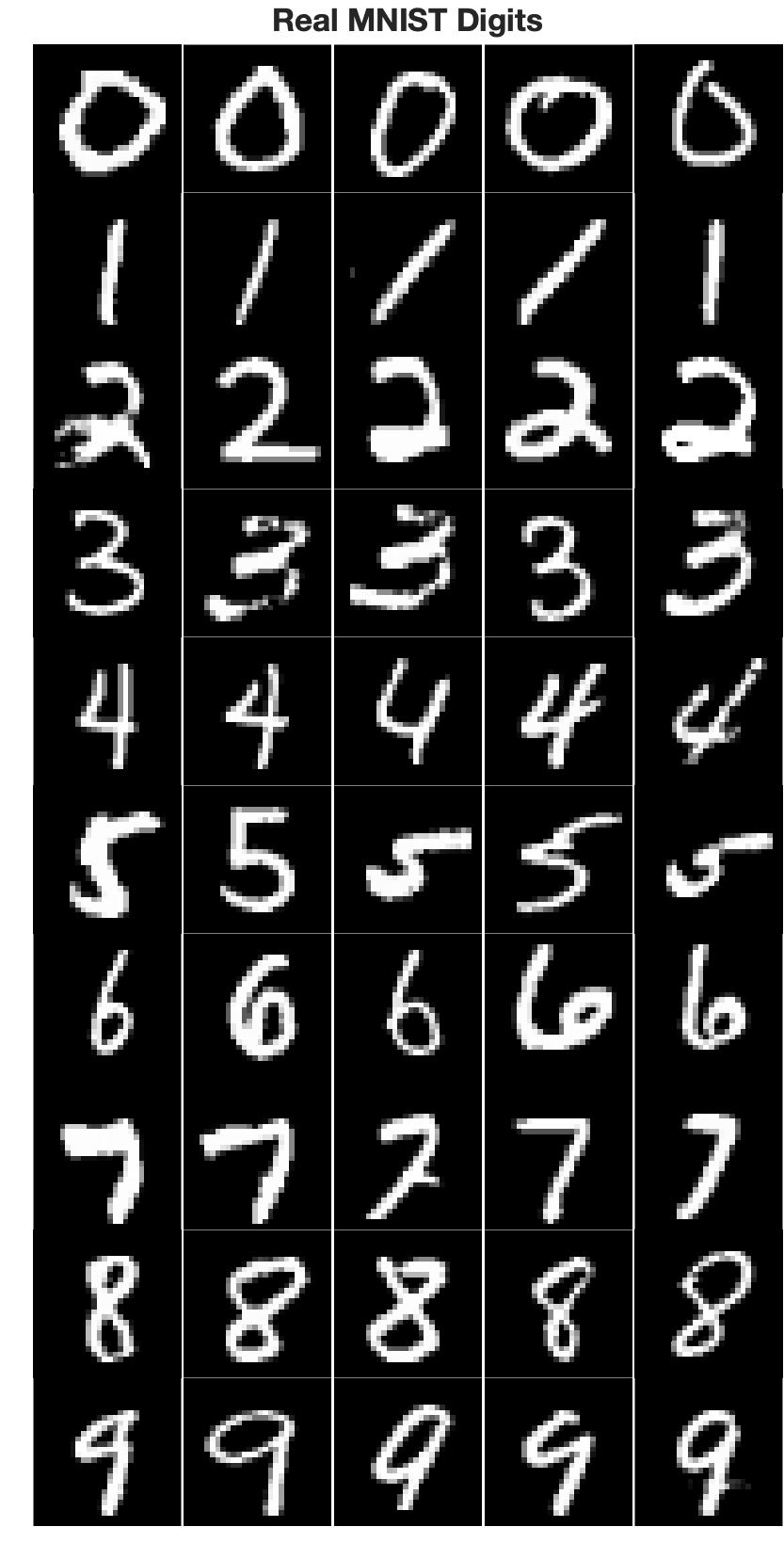}
    \caption{MNIST dataset result. (Left) Generated Images. (Right) Real Images. For each digit number, we choose five images randomly.}\label{fig: MNIST}
\end{figure}

\section{Conclusion}\label{sec: conclusion}

In this work, we introduced an alternative approach for diffusion models that avoids both iterative optimization and forward SDE. By representing the score function in a sparse set of eigenbasis of the backward Kolmogorov operator associated with a base distribution, we reformulated score estimation as a closed-form linear system problem—eliminating the need for iterative optimization and stochastic simulation. We provided a theoretical analysis of the approximation error using perturbation theory, showing that the score can be accurately estimated under mild assumptions. Our numerical experiments on high-dimensional Boltzmann distributions and real-world datasets, such as MNIST, demonstrate that the proposed method achieves competitive performance. As a direction for future work, we aim to extend this approach to more complex models and datasets. \par

\backmatter

\section*{Declarations}
\bmhead{Acknowledgements}
The authors acknowledge DMS-2339439 and DE-SC0022232, Sloan foundation, and funding by the Deutsche Forschungsgemeinschaft (DFG, German Research Foundation) -- project number 442047500 -- through the Collaborative Research Center ``Sparsity and Singular Structures'' (SFB 1481) for travel support. 
\\








\begin{appendices}

\section{Coefficient Computation in High-dimensional 
 Setting}\label{appendix: coefficient computation high-d}
 
In this appendix, we provide detailed computations of the high-dimensional coefficients involved in the linear system. Recall the definitions of the coefficient matrix $A(t) \in \mathbb{R}^{|\mathcal{S}| \times |\mathcal{S}|}$ and the vector $\bvec{b}^{(i)}(t) \in \mathbb{R}^{|\mathcal{S}| \times 1}$:

\begin{align}\label{eq: A and b appendix}
  & A_{l,l'}(t)=\int \rho_t (\bvec{x})f_l(\bvec{x})f_{l'}(\bvec{x})\mathrm{d}\bvec{x}, \nonumber\\
& b^{(i)}_l(t)=\int \rho_t(\bvec{x})\left(\partial_{x_i}f_l(\bvec{x})-\beta f_l(\bvec{x})\partial_{x_i}V(\bvec{x})\right)\mathrm{d}\bvec{x},
\end{align}
Each basis function $f_l(\bvec{x})$ is of the form $\phi^{(k_1)}_{n_{k_1}}(x_{k_1}) \phi^{(k_2)}_{n_{k_2}}(x_{k_2})$, for some pair of indices $(k_1, k_2)$.

We compute $A(t)$ first. Using the eigenfunction property, we have:
\begin{align*}
        A_{l,l'}(t) = \int \rho_t(\bvec{x})  \left(\prod_{j=1}^4 \phi^{(k_j)}_{n_{k_j}}(x_{k_j})\right) \mathrm{d}\bvec{x} 
         = \int \rho_0(\bvec{x}) e^{\mathcal{L} t}\left(\prod_{j=1}^4 \phi^{(k_j)}_{n_{k_j}}(x_{k_j})\right) \mathrm{d}\bvec{x} 
\end{align*}
where the product arises from multiplying the two basis functions associated with indices $l$ and $l'$. According to the local 2-cluster basis construction in equation~\eqref{eq: 2-cluster local basis}, we have $k_1 \neq k_2$ and $k_3 \neq k_4$. Based on this structure, we categorize the computations into several distinct cases.

\begin{enumerate}
    \item When all four indices are distinct ($k_1 \neq k_2 \neq k_3 \neq k_4$), the eigenfunctions are separable, and the operator $e^{\mathcal{L}t}$ can be applied to each term individually:
    \begin{equation*}
        e^{\mathcal{L} t}\left(\prod_{j=1}^4 \phi^{(k_j)}_{n_{k_j}}(x_{k_j})\right) = e^{(\sum_{j=1}^4\lambda^{(k_j)}_{n_{k_j}})t}\prod_{j=1}^4 \phi^{(k_j)}_{n_{k_j}}(x_{k_j})
    \end{equation*}
    Therefore,
    \begin{equation*}
        A_{l,l'}(t) = \int \rho_0(\bvec{x}) e^{\mathcal{L} t}\left(\prod_{j=1}^4 \phi^{(k_j)}_{n_{k_j}}(x_{k_j})\right) \mathrm{d}\bvec{x} = e^{(\sum_{j=1}^4\lambda^{(k_j)}_{n_{k_j}})t} \int \rho_0(\bvec{x}) \prod_{j=1}^4 \phi^{(k_j)}_{n_{k_j}}(x_{k_j}) \mathrm{d}\bvec{x}
    \end{equation*}
    The integral on the right-hand side is evaluated using Monte Carlo integration with samples drawn from $\rho_0$ at the initial time.

    \item When one pair among $\{x_{k_j}\}_{j=1}^4$ shares the same index (e.g., $k_1 = k_3 \neq k_2 \neq k_4$), we express the product of the two basis functions with the same variable as a linear combination in \eqref{eq:linear interpolation}:

    $\phi^{(k_1)}_{n_{k_1}}(x_{k_1})\phi^{(k_3)}_{n_{k_3}}(x_{k_3})= \sum_{l_1=1}^m u^{(n_{k_1},n_{k_3})}_{l_1}\phi^{(k_1)}_{l_1}(x_{k_1})$. Here $|\mathcal{S}'|=m$ follows the same expansion strategy demonstrated in Section~\ref{sec: solve LS 1d}. Therefore, 
    \begin{align*}
        &e^{\mathcal{L} t}\left(\prod_{j=1}^4 \phi^{(k_j)}_{n_{k_j}}(x_{k_j})\right) = e^{\mathcal{L}_{k_1} t}\left( \phi^{(k_1)}_{n_{k_1}}(x_{k_1})\phi^{(k_3)}_{n_{k_3}}(x_{k_3}) \right)
        e^{\mathcal{L}_{k_2} t}\phi^{(k_2)}_{n_{k_2}}(x_{k_2})
        e^{\mathcal{L}_{k_4} t}\phi^{(k_4)}_{n_{k_4}}(x_{k_4}) \nonumber\\
        & = e^{\mathcal{L}_{k_1} t}\left( \sum_{l_1=1}^m u^{(n_{k_1},n_{k_2})}_{l_1}\phi^{(k_1)}_{l_1}(x_{k_1}) \right)
        e^{\lambda^{(k_3)}_{n_{k_3}} t}\phi^{(k_3)}_{n_{k_3}}(x_{k_3})
        e^{\lambda^{(k_4)}_{n_{k_4}} t}\phi^{(k_4)}_{n_{k_4}}(x_{k_4}) \nonumber\\
        & = \left( \sum_{l_1=1}^m u^{(n_{k_1},n_{k_2})}_{l_1} e^{\lambda^{(k_1)}_{l_1} t} \phi^{(k_1)}_{l_1}(x_{k_1}) \right)
        e^{\lambda^{(k_3)}_{n_{k_3}} t}\phi^{(k_3)}_{n_{k_3}}(x_{k_3})
        e^{\lambda^{(k_4)}_{n_{k_4}} t}\phi^{(k_4)}_{n_{k_4}}(x_{k_4}) \nonumber\\
        & =  \sum_{l_1=1}^m u^{(n_{k_1},n_{k_2})}_{l_1} e^{(\lambda^{(k_1)}_{l_1}+\lambda^{(k_3)}_{k_3}+\lambda^{(k_4)}_{k_4}) t} \phi^{(k_1)}_{l_1}(x_{k_1})\phi^{(k_3)}_{n_{k_3}}(x_{k_3})\phi^{(k_4)}_{n_{k_4}}(x_{k_4})
    \end{align*}
    Hence,
    \begin{equation*}
        A_{l,l'}(t) = \sum_{l_1=1}^m u^{(n_{k_1},n_{k_2})}_{l_1} e^{(\lambda^{(k_1)}_{l_1}+\lambda^{(k_3)}_{k_3}+\lambda^{(k_4)}_{k_4}) t} 
 \int \rho_0(x)\phi^{(k_1)}_{l_1}(x_{k_1})\phi^{(k_3)}_{n_{k_3}}(x_{k_3})\phi^{(k_4)}_{n_{k_4}}(x_{k_4}) \mathrm{d}x
    \end{equation*}
    The integral on the right-hand side is evaluated using Monte Carlo integration, following the same approach as in the previous case.

    \item When there are two repeated pairs among $\{x_{k_j}\}_{j=1}^4$ (e.g., $k_1 = k_3$, $k_2 = k_4$, with $k_1 \neq k_2$), we express each product of basis functions with repeated indices as a linear combination:

    $\phi^{(k_1)}_{n_{k_1}}(x_{k_1})\phi^{(k_3)}_{n_{k_3}}(x_{k_3})= \sum_{l_1=1}^m u^{(n_{k_1},n_{k_3})}_{l_1}\phi^{(k_1)}_{l_1}(x_{k_1}), \phi^{(k_2)}_{n_{k_2}}(x_{k_2})\phi^{(k_4)}_{n_{k_4}}(x_{k_4})= \sum_{l_1=1}^m u^{(n_{k_2},n_{k_4})}_{l_1}\phi^{(k_2)}_{l_1}(x_{k_2})$, and thus
    \begin{align*}
        &e^{\mathcal{L} t}\left(\prod_{j=1}^4 \phi^{(k_j)}_{n_{k_j}}(x_{k_j})\right) = e^{\mathcal{L}_{k_1} t}\left( \phi^{(k_1)}_{n_{k_1}}(x_{k_1})\phi^{(k_3)}_{n_{k_3}}(x_{k_3}) \right)
        e^{\mathcal{L}_{k_2} t}\left(\phi^{(k_2)}_{n_{k_2}}(x_{k_2})\phi^{(k_4)}_{n_{k_4}}(x_{k_4})\right) \nonumber\\
        & = e^{\mathcal{L}_{k_1} t}\left( \sum_{l_1=1}^m u^{(n_{k_1},n_{k_3})}_{l_1}\phi^{(k_1)}_{l_1}(x_{k_1}) \right)e^{\mathcal{L}_{k_2} t}\left( \sum_{l_2=1}^m u^{(n_{k_2},n_{k_4})}_{l_2}\phi^{(k_2)}_{l_2}(x_{k_2})\right) \nonumber\\
        & = \left( \sum_{l_1=1}^m u^{(n_{k_1},n_{k_3})}_{l_1} e^{\lambda^{(k_1)}_{l_1} t} \phi^{(k_1)}_{l_1}(x_{k_1}) \right)
        \left( \sum_{l_2=1}^m u^{(n_{k_2},n_{k_4})}_{l_2}e^{\lambda^{(k_2)}_{l_2} t} \phi^{(k_2)}_{l_2}(x_{k_2})\right)\nonumber\\
        & =  \sum_{l_1,l_2=1}^m  u^{(n_{k_1},n_{k_3})}_{l_1} u^{(n_{k_2},n_{k_4})}_{l_2} e^{(\lambda^{(k_1)}_{l_1}+\lambda^{(k_2)}_{l_2}) t} \phi^{(k_1)}_{l_1}(x_{k_1}) \phi^{(k_2)}_{l_2}(x_{k_2})
    \end{align*} 
    Therefore, 
    \begin{equation*}
        A_{l,l'}(t) = \sum_{l_1,l_2=1}^m  u^{(n_{k_1},n_{k_3})}_{l_1} u^{(n_{k_2},n_{k_4})}_{l_2} e^{(\lambda^{(k_1)}_{l_1}+\lambda^{(k_2)}_{l_2}) t} \int \rho_0(x)\phi^{(k_1)}_{l_1}(x_{k_1}) \phi^{(k_2)}_{l_2}(x_{k_2})\mathrm{d}x
    \end{equation*}
\end{enumerate}
The case-by-case discussion above completes the computation of all entries in the matrix $A(t)$.\par

Next, we discuss the computation of $b^{(i)}(t)$. Since the chosen potential $V$ is of mean-field form, the partial derivative $\partial_{x_i} V(\bvec{x})$ depends only on the single variable $x_i$. Suppose we have the following linear expansions as in \eqref{eq:linear interpolation}:
\begin{equation}
    \phi^{(i)'}_{n_i}(x_{i}) = \sum_{l_1=1}^m v_{l_1}^{(n_i)} \phi^{(i)}_{l_1} (x_i), \partial_{x_i} V(x) = \sum_{l_1=1}^m q^{(i)}_{l_1} \phi^{(i)}_{l_1} (x_i)
\end{equation}

Assume that the basis function $f_l(\bvec{x})$ in equation~\eqref{eq: A and b appendix} takes the form $f_l(\bvec{x}) = \phi^{(k_1)}_{n_{k_1}}(x_{k_1}) \phi^{(k_2)}_{n_{k_2}}(x_{k_2})$. Then we have:
\begin{align*}
    e^{\mathcal{L}t}\left(\partial_{x_i}f_l(\bvec{x})-\beta f_l(\bvec{x})\partial_{x_i}V(\bvec{x})\right) = e^{\mathcal{L}t}\left(\partial_{x_i}(\phi^{(k_1)}_{n_{k_1}}(x_{k_1})\phi^{(k_2)}_{n_{k_2}}(x_{k_2}))-\beta \phi^{(k_1)}_{n_{k_1}}(x_{k_1})\phi^{(k_2)}_{n_{k_2}}(x_{k_2})\partial_{x_i}V(\bvec{x})\right)
\end{align*}
We now compute the expression above on a case-by-case basis, depending on the relationship between $i$, $k_1$, and $k_2$.\par 
\begin{enumerate}
    \item If $k_1 \neq i$ and $k_2 \neq i$, then the derivative term $\partial_{x_i} f_l(\bvec{x})$ vanishes. The expression simplifies to:
    \begin{align*}
        & -\beta e^{\mathcal{L}t}\left( \phi^{(k_1)}_{n_{k_1}}(x_{k_1})\phi^{(k_2)}_{n_{k_2}}(x_{k_2}) \partial_{x_i}V(\bvec{x}) \right) = -\beta e^{\mathcal{L}t}\left( \phi^{(k_1)}_{n_{k_1}}(x_{k_1})\phi^{(k_2)}_{n_{k_2}}(x_{k_2}) \sum_{l_1=1}^m q^{(i)}_{l_1} \phi^{(i)}_{l_1} (x_i)\right) \nonumber\\ 
        & =-\beta \sum_{l_1=1}^m q^{(i)}_{l_1} e^{\mathcal{L}t}\left( \phi^{(k_1)}_{n_{k_1}}(x_{k_1})\phi^{(k_2)}_{n_{k_2}}(x_{k_2})\phi^{(i)}_{l_1} (x_i)\right) \nonumber\\
        & = -\beta \sum_{l_1=1}^m q^{(i)}_{l_1} e^{(\lambda^{(k_1)}_{n_{k_1}} + \lambda^{(k_2)}_{n_{k_2}} + \lambda^{(i)}_{l_1}) t}\phi^{(k_1)}_{n_{k_1}}(x_{k_1})\phi^{(k_2)}_{n_{k_2}}(x_{k_2})\phi^{(i)}_{l_1} (x_i)
    \end{align*}
    Therefore,
    \begin{equation*}
        b_l^{(i)}(t) = -\beta \sum_{l_1=1}^m q^{(i)}_{l_1} e^{(\lambda^{(k_1)}_{n_{k_1}} + \lambda^{(k_2)}_{n_{k_2}} + \lambda^{(i)}_{l_1}) t} \int \rho_0(x)\phi^{(k_1)}_{n_{k_1}}(x_{k_1})\phi^{(k_2)}_{n_{k_2}}(x_{k_2})\phi^{(i)}_{l_1} (x_i) \mathrm{d}x
    \end{equation*}

    \item If $k_1=i$ and $k_2\neq k_1$, then the formula becomes
    \begin{align*}
         &e^{\mathcal{L}t}\left(\phi^{(i)'}_{n_{i}}(x_{i})\phi^{(k_2)}_{n_{k_2}}(x_{k_2}) -\beta \phi^{(i)}_{n_{i}}(x_{i})\phi^{(k_2)}_{n_{k_2}}(x_{k_2})\partial_{x_i}V(\bvec{x})\right) \nonumber\\
        & = e^{\mathcal{L}t}\left((\sum_{l_1=1}^m v_{l_1}^{(n_i)} \phi^{(i)}_{l_1} (x_i))  \phi^{(k_2)}_{n_{k_2}}(x_{k_2}) -\beta \phi^{(i)}_{n_{i}}(x_{i})\phi^{(k_2)}_{n_{k_2}}(x_{k_2})(\sum_{l_1=1}^m q^{(i)}_{l_1} \phi^{(i)}_{l_1} (x_i))\right) \nonumber\\
        & = \sum_{l_1=1}^m v_{l_1}^{(n_i)} e^{(\lambda^{(i)}_{l_1} + \lambda^{(k_2)}_{n_{k_2}})t } \phi^{(i)}_{l_1} (x_i)  \phi^{(k_2)}_{n_{k_2}}(x_{k_2}) -\beta \sum_{l_1=1}^m q^{(i)}_{l_1} e^{\mathcal{L}t}\left(\phi^{(i)}_{n_{i}}(x_{i})\phi^{(i)}_{l_1} (x_i)\phi^{(k_2)}_{n_{k_2}}(x_{k_2})\right) \nonumber\\
        & = \sum_{l_1=1}^m v_{l_1}^{(n_i)} e^{(\lambda^{(i)}_{l_1} + \lambda^{(k_2)}_{n_{k_2}})t } \phi^{(i)}_{l_1} (x_i)  \phi^{(k_2)}_{n_{k_2}}(x_{k_2}) -\beta \sum_{l_1=1}^m q^{(i)}_{l_1} e^{\mathcal{L}t}\left(\sum_{l_2=1}^m u^{(n_{i},l_1)}_{l_2}\phi^{(i)}_{l_2}(x_{i}) \phi^{(k_2)}_{n_{k_2}}(x_{k_2})\right) \nonumber\\
        & = \sum_{l_1=1}^m v_{l_1}^{(n_i)} e^{(\lambda^{(i)}_{l_1} + \lambda^{(k_2)}_{n_{k_2}})t } \phi^{(i)}_{l_1} (x_i)  \phi^{(k_2)}_{n_{k_2}}(x_{k_2}) -\beta \sum_{l_1=1}^m q^{(i)}_{l_1} \sum_{l_2=1}^m u^{(n_{i},l_1)}_{l_2} e^{(\lambda^{(i)}_{l_2}+\lambda^{(k_2)}_{n_{k_2}})t}\left(\phi^{(i)}_{l_2}(x_{i}) \phi^{(k_2)}_{n_{k_2}}(x_{k_2})\right)
    \end{align*}
    Therefore,
    \begin{align*}
       b_l^{(i)}(t) = & \sum_{l_1=1}^m v_{l_1}^{(n_i)} e^{(\lambda^{(i)}_{l_1} + \lambda^{(k_2)}_{n_{k_2}})t } \int \rho_0(x)\phi^{(i)}_{l_1} (x_i)  \phi^{(k_2)}_{n_{k_2}}(x_{k_2}) \mathrm{d}x \nonumber\\
       &-\beta \sum_{l_2=1}^m\left(\sum_{l_1=1}^m q^{(i)}_{l_1}  u^{(n_{i},l_1)}_{l_2}\right) e^{(\lambda^{(i)}_{l_2}+\lambda^{(k_2)}_{n_{k_2}})t}\int \rho_0(x)\phi^{(i)}_{l_2}(x_{i}) \phi^{(k_2)}_{n_{k_2}}(x_{k_2}) \mathrm{d}x
    \end{align*}
    \item If $k_2=i$ and $k_2\neq k_1$, the case is equivalent to the second case due to symmetry.  
\end{enumerate}

We summarize the computational complexity involved in computing $A(t)$ and $b^{(i)}(t)$. Note that there are $O(|\mathcal{S}|^2)$ integral terms of the form
\[
\int \rho_0(\bvec{x}) \prod_{j=1}^4 \phi^{(k_j)}_{n_{k_j}}(x_{k_j}) \, \mathrm{d}\bvec{x},
\]
which dominate the cost of the initial Monte Carlo integration step. This results in a computational cost of $O(N d_{\text{b}}^2 d^2 n^4)$. Assume $m=O(n)$. \par 

\par
For the computation of $A(t)$:
\begin{itemize}
    \item \textbf{Case 1}: There are $O(|\mathcal{S}|^2)$ terms, leading to a cost of $O(d_{\text{b}}^2 d^2 n^4)$.
    \item \textbf{Case 2}: The number of relevant dimension index pairs is $O(d_{\text{b}} d d_{\text{b}})$. Taking into account the basis expansion, the total cost is $O(d_{\text{b}} d d_{\text{b}} m n^2) = O(d_{\text{b}}^2 d n^3)$.
    \item \textbf{Case 3}: The total cost is $O(d_{\text{b}} d m^2) = O(d_{\text{b}} d n^2)$.
\end{itemize}
Summing over all cases, the overall computational complexity for $A(t)$ remains dominated by $O(d_{\text{b}}^2 d^2 n^4)$.

\par
For the computation of $b^{(i)}(t)$:
\begin{itemize}
    \item \textbf{Case 1}: The cost is $O(d d_{\text{b}} d n^2 m) = O(d_{\text{b}} d^2 n^3)$.
    \item \textbf{Case 2}: The cost is $O(d d_{\text{b}} n m^2) = O(d_{\text{b}} d n^3)$.
\end{itemize}
As a result, the computational cost for evaluating $b^{(i)}(t)$ is smaller than that of $A(t)$. The total computational cost becomes $O(d_{\text{b}}^2 d^2 n^4)$.

\section{Example for Assumption~\ref{ass:measure}}\label{appendix: example}
In this appendix, we provide two one-dimensional examples and one high-dimensional example to illustrate Assumption~\ref{ass:measure}. \par

\begin{example}
In the case of the Ornstein–Uhlenbeck process, the stationary distribution $\rho_\infty$ is Gaussian. We choose $f_n$ to be the eigenfunctions associated with $\rho_\infty$, which correspond to Hermite polynomials $H_n$. Using this basis, we can expand the following two terms in terms of Hermite polynomials:

\begin{equation*}
       e^{-a^2x^2} = \sum_{i=0}^\infty \frac{(-1)^i a^{2i}}{i!(1+a^2)^{i+\frac 1 2}2^{2i}}H_{2i}(x) 
\end{equation*}
    and
\begin{equation*}
e^{bx}  = e^{\frac{b^2}4}\sum_{i=0}^\infty \frac{b^i}{i!2^i}H_i(x)     
\end{equation*}
for some $a$ and $b$. Besides, the ratio of two Gaussian distribution $\rho_0/\rho_\infty$ is still a Gaussian distribution. Therefore, we can represent it by Hermite polynomials. Besides, when choosing appropriate $\delta$, we have $\sup_{x\in\mathbb R}|\sum_{i=1}^\infty p_i f_i(x)|\leq 1$.\par 
Next we verify the second assumption. For Hermite polynomial, we have $H_i'(x) = iH_{i-1}(x)$. We can also assume $|p_i|\leq \gamma^i$ for $\gamma<1$ according to the previous expansion. Then we get
    \begin{align*}
        \nabla \sum_{i=M}^\infty p_i H_i(x) = \sum_{i=M}^\infty ip_i H_{i-1}(x)
    \end{align*}
and thus
    \begin{align*}
        \int(\nabla\sum_{i=M}^\infty p_i H_i(x))^2\rho_0(x)\mathrm{d}x\leq& C_\rho(1+\delta)\int (\sum_{i=M}^\infty i p_i H_{i-1}(x))^2\rho_{\infty}(x)\mathrm{d}x \nonumber\\
        \leq& C_\rho(1+\delta)\sum_{i,j=M}^\infty ij |p_i||p_j|\int H_{i-1}(x)H_{j-1}(x)\rho_\infty(x) \mathrm{d}x \nonumber\\
        =&C_\rho(1+\delta) \sum_{i=M}^\infty i^2 |p_i|^2
        \leq C_\rho(1+\delta) \sum_{i=M}^\infty i^2\gamma^{2i} \nonumber\\
        \leq & C_\rho(1+\delta)\gamma^{2(M-1)}\left(\frac{(M-1)^2}{2\ln{1/\gamma}} + \frac{M-1}{2{\ln \gamma}^2} + \frac{1}{4{\ln \gamma}^3} \right) =: r(M)
    \end{align*}
We denote the right hand side as $r(M)$. Due to $\gamma<1$, the function converges to zero as $M\rightarrow\infty$.  
\end{example}

\begin{example} 
    In Fourier case defined in the domain $[-L,L]$, then $\rho_\infty$ is a constant. We expand a Gaussian function $\rho_0=\exp(-a^2x^2)$ as the following
    \begin{equation*}
        \rho_0  = C_\rho \rho_\infty(p_0 + \sum_{k=1}^\infty p_k \cos(\frac{\pi k x}{L}))
    \end{equation*}
    where $p_k = \frac{2}{C_\rho}\int_0^L e^{-a^2x^2}\cos(\frac{\pi k x}{L}) \mathrm{d}x \sim C\exp(-\frac{\pi k^2}{4a^2L^2})$ as $k\rightarrow\infty$, we can assume $|p_k|\leq \gamma^k$ for some $\gamma<1$ when $a$ is small. Thus, the condition $\sup_{x\in\mathbb R}| \sum_{i=1}^\infty p_i f_i(x)|\leq 1$ is straightforward. \par 
    Besides, we follow the same logic as in the Hermite polynomial case and obtain
    \begin{align*}
        \int(\nabla\sum_{k=M}^\infty p_k \cos(\frac{\pi k x}{L}))^2\rho_0(x)\mathrm{d}x\leq& C_\rho(1+\delta)\int (\sum_{k=M}^\infty \frac{\pi k}{L} p_k \sin(\frac{\pi k x}{L}))^2\rho_{\infty}(x)\mathrm{d}x \nonumber\\
        \leq& C_\rho(1+\delta)\sum_{k,k'=M}^\infty \frac{\pi^2}{L^2} kk' |p_k||p_{k'}|\int \sin(\frac{\pi k x}{L})\sin(\frac{\pi k' x}{L})\rho_\infty(x) \mathrm{d}x \nonumber\\
        =& C_\rho(1+\delta)\sum_{k=M}^\infty \frac{\pi^2k^2}{L^2}  |p_k|^2
        \leq  C_\rho(1+\delta)\frac{\pi^2}{L^2} \sum_{k=M}^\infty   k^2\gamma^{2k} \nonumber\\
        \leq & C_\rho(1+\delta)\frac{\pi^2}{L^2}\gamma^{2(M-1)}\left(\frac{(M-1)^2}{2\ln{1/\gamma}} + \frac{M-1}{2{\ln \gamma}^2} + \frac{1}{4{\ln \gamma}^3} \right) =: r(M)
    \end{align*}
\end{example}

\begin{example} We are in particular interested in high-dimensional potentials $V$ that exhibit some "nearest-neighbor" interactions, commonly found in physics.
For example let 
    \begin{align*}
        \rho_0 =&\left(\prod_{i=1}^d \mu_i(x_i)\right)\exp\left(\delta \sum_{0<|i-j|\leq d_{\text{b}}}\alpha_{ij}x_ix_j\right) \nonumber\\
        =&\left(\prod_{i=1}^d \mu_i(x_i)\right)\prod_{0<|i-j|\leq d_{\text{b}}}\exp\left({\delta\alpha_{ij}x_ix_j}\right)
    \end{align*}
    where $\mu_i(x_i)$ is the mean-field term of the density and $\exp\left({\delta\alpha_{ij}x_ix_j}\right)$ is the interaction term. 
 
    Suppose $\delta$ is small, which corresponds to a high-temperature scenario. We can expand $\exp({\delta\alpha_{ij}x_ix_j})$ in terms of $\delta$:
    \begin{align*}
        \exp(\delta\alpha_{ij}x_ix_j) = 1 + \sum_{k=1}^\infty \delta^k\frac{\alpha_{ij}^k}{k!}x_i^rx_j^k = 1 + \delta\sum_{k=1}^\infty \delta^{k-1}\frac{\alpha_{ij}^k}{k!}x_i^k x_j^k.
    \end{align*}
    Assume $\{f^{(i)}_{n_i}\}_{n\in\mathbb N}$ are the eigenfunctions associated to mean-field density $\mu_i(x_i)$. Now we can expand $\sum_{k=1}^\infty \delta^{k-1}\frac{\alpha_{ij}^k}{k!}x_i^k x_j^k$ in terms of eigenfunctions and thus, we have
    \begin{align*}
        \exp({\delta \alpha_{ij}x_ix_j}) = &\left(1+\delta\sum_{n_i,n_j=1}^\infty \gamma^{i,j}_{n_i,n_j}(\delta) f^{(i)}_{n_i}(x_i)f^{(j)}_{n_j}(x_j)\right)
    \end{align*}
    where we have that coefficient $\gamma_{n_i,n_j}^{i,j}(\delta)\to 0$ as $\delta\to 0$. 
    
    Let us call $\sum_{n_i,n_j=1}^\infty \gamma^{i,j}_{n_i,n_j}(\delta) f^{(i)}_{n_i}(x_i)f^{(j)}_{n_j}(x_j)=C_2(i,j)$ as a $2$-cluster of pair $(i,j)$. Then 
    \begin{equation*}
        \exp({-\delta \alpha_{ij}x_ix_j}) = 1+\delta C_2(i,j)
    \end{equation*} 
    and thus 
    \begin{equation*}
        \exp\left({\delta\sum_{i,j=1}^d \alpha_{ij}x_ix_j}\right) = \prod_{i,j=1}^d \left(1+\delta C_2(i,j\right)) = 1+\delta\,C_2 +\delta^2C_{>2}
    \end{equation*}
    Where $C_2$ is the sum of all 2-cluster and this is also a 2-cluster, while $C_{>2}$ contains all products of 2-clusters and thus is a higher-order cluster. 
    
    Therefore, we can interpret $\rho_0$ as 2-cluster perturbation of $\rho_\infty$ since
    \begin{equation*}
        \rho_0 = \rho_\infty(1+\delta\,C_2 +\delta^2C_{>2})
    \end{equation*}
    Based on Theorem \ref{thm:General} we can expand the score up to first order in $\delta$ as a 2-cluster function.
\end{example}

\section{Proof of Theorem~\ref{thm:Fourier}}\label{appendix: proof}

In order to proof the theorem we need the following lemma.
\begin{lemma}\label{lem:ApproInx}
Denote Fourier basis $f_k(x) = e^{\sqrt{-1}wkx}$. For $\varepsilon>0$ there is $M$ and $g\in \mathrm{span}\{f_{-M},\dots,f_M\}$ such that 
    \begin{align*}
    \left|\frac{1}{\displaystyle1+\delta \sum_{i=-\infty}^\infty p_i e^{\lambda_i t}f_i(x)}\,-  (1+g(x))\right|\leq \varepsilon e^{\lambda_1{t}}
\end{align*}
where we can choose $M\in\mathbb N$ to satisfy
    \begin{align*}
        \gamma^{M}\leq \frac \varepsilon{8\frac{\log(\frac{\varepsilon}2 (1-\delta))}{\log\delta}} \frac{(1-\gamma)(1-\delta)}{\delta^2(1-\delta^{\frac{\log(\frac{\varepsilon}2 (1-\delta))}{\log\delta}})}
    \end{align*}
as well as \begin{align*}
    (2M)^{\frac{\log(\frac{\varepsilon}2 (1-\delta))}{\log\delta}}\gamma^M\leq \frac{\varepsilon}{8(1+\frac{\log(\frac{\varepsilon}2 (1-\delta))}{\log\delta})}.
\end{align*}
\end{lemma}

We put the proof of this lemma in the end. Now we go to Proof of Theorem \ref{thm:Fourier} first. 
\begin{proof}[Proof of Theorem \ref{thm:Fourier}] In the one-dimensional case, 
\begin{align*}
    \partial_x\log\rho_t(x)  =& \partial_x  \log(1+\delta\sum_{i=1}^\infty p_ie^{\lambda_i t} f_i(x))\nonumber\\
    =& \frac{  \delta\sum_{i=1}^\infty p_ie^{\lambda_i t}  f'_i(x)}{1+\delta\sum_{i=1}^\infty p_ie^{\lambda_i t} f_i(x)}
\end{align*}
Therefore, for some $\highds_{2M}\in F_{2M}$, we have
    \begin{align*}
       & \int \frac{1}{2}(\partial_x\log\rho_t(x)-\highds_{2M}(t,x))^2\rho_{t}(x)\mathrm{d}x = 
       \int (\frac{  \delta\sum_{i=1}^\infty p_ie^{\lambda_i t}  f'_i(x)}{1+\delta\sum_{i=1}^\infty p_ie^{\lambda_i t} f_i(x)}-\highds_{2M}(t,x))^2\rho_{t}(x)\mathrm{d}x \nonumber\\
       & =  \int \frac{1}{2}(\frac{  \delta\sum_{i=M+1}^\infty p_ie^{\lambda_i t}  f'_i(x)}{1+\delta\sum_{i=1}^\infty p_ie^{\lambda_i t} f_i(x)} + \frac{  \delta\sum_{i=1}^M p_ie^{\lambda_i t}  f'_i(x)}{1+\delta\sum_{i=1}^\infty p_ie^{\lambda_i t} f_i(x)}-\highds_{2M}(t,x))^2\rho_{t}(x)\mathrm{d}x \nonumber\\
       & \leq \int (\frac{  \delta\sum_{i=M+1}^\infty p_ie^{\lambda_i t}  f'_i(x)}{1+\delta\sum_{i=1}^\infty p_ie^{\lambda_i t} f_i(x)})^2\rho_{t}(x) + (\frac{  \delta\sum_{i=1}^M p_ie^{\lambda_i t}  f'_i(x)}{1+\delta\sum_{i=1}^\infty p_ie^{\lambda_i t} f_i(x)}-\highds_{2M}(t,x))^2\rho_{t}(x)\mathrm{d}x 
    \end{align*}
    For the first term, note that $(\frac{1}{1+\delta\sum_{i=1}^\infty p_ie^{\lambda_i t} f_i(x)})^2\leq \frac{1}{(1-\delta)^2}$ and  \begin{equation*}
        \| \sum_{i=M+1}^\infty p_i f'_i(x)\|^2_{L^2(\rho_0)}\leq r(M+1)
    \end{equation*}
     shown in the assumption~\ref{ass:measure}. Thus, there is a $M$ such that 
     \begin{equation*}
         (\frac{\delta}{1-\delta})^2\|\sum_{i={M+1}}^\infty p_i f'_i(x)\|^2_{L^2(\rho_0)} \leq \frac{\varepsilon}{2}.
     \end{equation*}
     Therefore,
     \begin{equation*}
         \int (\frac{  \delta\sum_{i=M+1}^\infty p_ie^{\lambda_i t}  f'_i(x)}{1+\delta\sum_{i=1}^\infty p_ie^{\lambda_i t} f_i(x)})^2\rho_{t}(x) \mathrm{d}x \leq \frac{\epsilon}{2}e^{2\lambda_{M+1}t}
     \end{equation*}

     For the second term, we follow Lemma \ref{lem:ApproInx} choosing $\tilde\varepsilon=\frac{\varepsilon}{2\|\sum_{i=1}^\infty p_if'_i\|^2_{L^2(\rho_0)}}$ we have that there is $\tilde{s}_M\in\mathrm{span}\{f_{-M},\dots,f_M\}$ such that
    \begin{align*}
        \left|\frac{\delta\sum_{i=1}^M p_ie^{\lambda_it}f'_i(x)}{1+\delta \sum_{i=1}^\infty p_i e^{\lambda_it} f_i(x)}- (\delta\sum_{i=1}^M p_ie^{\lambda_it}f'_i(x))(1+\tilde{s}_M(t,x))\right|\leq \frac{\varepsilon}{2}e^{\lambda_1t}.
    \end{align*}
    By observing that $f_n'(x) = \sqrt{-1}n\omega f_n(x)$ and $f_n(x)f_m(x)=f_{n+m}(x)$ we can expand 
    \begin{equation*}
        (\delta\sum_{i=1}^Mp_ie^{\lambda_i t}f'_i(x))(1+\tilde{s}_M(t,x)):=\highds_{2M}(t,x)
    \end{equation*} 
    in Fourier basis from $-2M$ to $2M$. Denote it as $\highds_{2M}(t,x)$. Therefore, 
    \begin{equation*}
        \int  (\frac{  \delta\sum_{i=1}^M p_ie^{\lambda_i t}  f'_i(x)}{1+\delta\sum_{i=1}^\infty p_ie^{\lambda_i t} f_i(x)}-\highds_{2M}(t,x))^2\rho_{t}(x)\mathrm{d}x  \leq \frac{\varepsilon}{2}e^{\lambda_1t}
    \end{equation*}
In the end, 
\begin{equation*}
    \int \frac{1}{2}(\partial_x\log\rho_t(x)-\highds_{2M}(t,x))^2\rho_{t}(x)\mathrm{d}x \leq \frac{\epsilon}{2}e^{2\lambda_{M+1}t} +\frac{\varepsilon}{2}e^{\lambda_1t} \leq \varepsilon e^{\lambda_1t}
\end{equation*}
and
\begin{equation*}
    \int_0^\infty \int \frac{1}{2}(\partial_x\log\rho_t(x)-\highds_{2M}(t,x))^2\rho_{t}(x)\mathrm{d}x \mathrm{d}t\leq \frac{\varepsilon}{|\lambda_1|}
\end{equation*}

\end{proof}

\begin{proof}[Proof of Lemma \ref{lem:ApproInx}]
\label{proof:LemApproxInx}
To ease the notation we assume $p_0=0$ according to our assumption. Let $\varepsilon>0$. The proof of this theorem is based is on a more refined analysis of the expansion of \begin{equation*}
    \frac{1}{1+\sum_{i=-\infty}^\infty p_i e^{\lambda_i t}f_i(x)}.
\end{equation*}

Recall that $\frac 1{1+x} = \sum_{k=0}^\infty (-1)^k x^k$ for $|x|<1$. Thus, $\sum_{k=N_1}^\infty |x|^k\leq \frac{|x|^{N_1}}{1-|x|}.$
With the same arguments as in the proof of Theorem \ref{thm:General} we can choose $N_1$ such that
Let $N_1$ be such that $$\frac{\delta^{N_1+1}}{1-\delta}\leq \frac\varepsilon 2e^{\lambda_1t}$$

We can split
\begin{align*}
    &\frac{1}{\displaystyle1+\delta \sum_{i=-\infty}^\infty p_i e^{\lambda_i t}f_i(x)} \nonumber\\
    =& \sum_{k=0}^{N_1} (-1)^k\left(\delta\sum_{i=-\infty}^\infty p_ie^{\lambda_i t} f_i(x) \right)^k+\sum_{k=N_1+1}^{\infty} (-1)^k\left(\delta\sum_{i=-\infty}^\infty p_ie^{\lambda_i t} f_i(x) \right)^k
\end{align*}

Then, 

\begin{align*}
    &\sum_{k=1}^{N_1} (-1)^k\left(\delta\sum_{i=\infty}^\infty p_ie^{\lambda_i t} f_i(x) \right)^k\\
    =& \sum_{k=1}^{N_1} (-1)^k \sum_{r=0}^k\begin{pmatrix}
        k\\r
    \end{pmatrix}\left( \sum_{i=-N_2}^{N_2}\delta e^{\lambda_it} p_i f_i(x)\right)^{k-r}\left( \sum_{i\in \mathbb N\setminus [-N_2,N_2]}\delta e^{\lambda_it}p_if_i(x)\right)^{r}\\
    =& \sum_{k=1}^{N_1} (-1)^k \left( \sum_{i=-N_2}^{N_2}\delta e^{\lambda_it} p_i f_i(x)\right)^{k}\\
    &+ \sum_{k=1}^{N_1} (-1)^k \sum_{r=1}^k\begin{pmatrix}
        k\\r
    \end{pmatrix}\left( \sum_{i=-N_2}^{N_2}\delta e^{\lambda_it} p_i f_i(x)\right)^{k-r}\left( \sum_{i\in \mathbb N\setminus [-N_2,N_2]}\delta e^{\lambda_it}p_i f_i(x)\right)^{r} \\
    \end{align*}
Observe
\begin{align*}
    &\sum_{k=1}^{N_1} (-1)^k \sum_{r=1}^k\begin{pmatrix}
        k\\r
    \end{pmatrix}\left( \sum_{i=-N_2}^{N_2}\delta e^{\lambda_it}p_i f_i(x)\right)^{k-r}\left( \sum_{i\in \mathbb N\setminus [-N_2,N_2]}\delta e^{\lambda_it}p_i f_i(x)\right)^{r}\\
    =&\left( \sum_{i\in \mathbb N\setminus [-N_2,N_2]}\delta e^{\lambda_it}p_i f_i(x)\right)\cdot\\
    &\qquad\sum_{k=1}^{N_1} (-1)^k \sum_{r=1}^k\begin{pmatrix}
        k\\r
    \end{pmatrix}\left( \sum_{i=-N_2}^{N_2}\delta e^{\lambda_it}p_i f_i(x)\right)^{k-r}\left( \sum_{i\in \mathbb N\setminus [-N_2,N_2]}\delta e^{\lambda_it}p_i f_i(x)\right)^{r-1}\\
    =&\left( \sum_{i\in \mathbb N\setminus [-N_2,N_2]}\delta e^{\lambda_it}p_i f_i(x)\right)\cdot\\
    &\qquad\sum_{k=1}^{N_1} (-1)^k \sum_{r=0}^{k-1}\begin{pmatrix}
        k\\r+1
    \end{pmatrix}\left( \sum_{i=-N_2}^{N_2}\delta e^{\lambda_it}p_i f_i(x)\right)^{k-1-r}\left( \sum_{i\in \mathbb N\setminus [-N_2,N_2]}\delta e^{\lambda_it}p_i f_i(x)\right)^{r}\\
        =&\left( \sum_{i\in \mathbb N\setminus [-N_2,N_2]}\delta e^{\lambda_it}p_i f_i(x)\right)\cdot\\
        &\qquad\sum_{k=1}^{N_1} (-1)^k \sum_{r=0}^{k-1}\frac{k}{r+1}\begin{pmatrix}
    k-1\\r
\end{pmatrix}\left( \sum_{i=-N_2}^{N_2}\delta e^{\lambda_it}p_i f_i(x)\right)^{k-1-r}\left( \sum_{i\in \mathbb N\setminus [-N_2,N_2]}\delta e^{\lambda_it}p_i f_i(x)\right)^{r}\\
\end{align*}
using
$\begin{pmatrix}
    k\\r+1
\end{pmatrix} =\frac{k}{r+1}\begin{pmatrix}
    k-1\\r
\end{pmatrix}$
and thus
\begin{align*}
    &|\sum_{k=1}^{N_1} (-1)^k \sum_{r=1}^k\begin{pmatrix}
        k\\r
    \end{pmatrix}\left( \sum_{i=-N_2}^{N_2}\delta e^{\lambda_it}p_i f_i(x)\right)^{k-r}\left( \sum_{i\in \mathbb N\setminus [-N_2,N_2]}\delta e^{\lambda_it}p_i f_i(x)\right)^{r}|\\
    \leq &|\sum_{i\in \mathbb N\setminus [-N_2,N_2]}\delta e^{\lambda_it}p_i f_i(x)|\sum_{k=1}^{N_1}  \sum_{r=0}^{k-1}\frac{k}{r+1}\begin{pmatrix}
    k-1\\r
\end{pmatrix}\left( \sum_{i=-N_2}^{N_2}\delta |p_i|\right)^{k-1-r}\left( \sum_{i\in \mathbb N\setminus [-N_2,N_2]}\delta |p_i| \right)^{r}\\
\leq& N_1 |\sum_{i\in \mathbb N\setminus [-N_2,N_2]}\delta e^{\lambda_it}p_i f_i(x)|\sum_{k=1}^{N_1}  \left( \sum_{i=-\infty}^{\infty}\delta |p_i|\right)^{k}\\
\leq & N_1  \frac{\delta(1-\delta^{N_1})}{1-\delta}|\sum_{i\in \mathbb N\setminus [-N_2,N_2]}\delta e^{\lambda_it}p_i f_i(x)|\\
\leq & N_1 e^{\lambda_1t}\frac{\delta^2(1-\delta^{N_1})}{1-\delta}\sum_{i\in \mathbb N\setminus [-N_2,N_2]} | p_i|\\
\leq &2N_1 e^{\lambda_1t}\frac{\delta^2(1-\delta^{N_1})}{1-\delta} \sum_{i=N_2}^\infty |p_i|\\
\leq &2N_1 e^{\lambda_1t}\frac{\delta^2(1-\delta^{N_1})}{1-\delta} \frac{\gamma^{N_2}}{1-\gamma}
\end{align*}
    
 Furthermore,   
    \begin{align*}
    &\sum_{k=1}^{N_1} (-1)^k \left( \sum_{i=-N_2}^{N_2}\delta e^{\lambda_it}p_i f_i(x)\right)^{k}\\
     =&\sum_{k=1}^{N_1}((-\delta)^k \sum_{\sum_{-N_2}^{N_2}j_i=k}^{N_2} \begin{pmatrix}
         k\\j_{-N_2}\dots j_{N_2}
     \end{pmatrix} \Pi_{i=-N_2}^{N_2}  e^{j_i\lambda_it}p_i^{j_i}  f_{\sum_{i=-N_2}^{N_2} j_i\cdot i}(x)\\
          =&\sum_{k=1}^{N_1}((-\delta)^k \sum_{m=-N_2k}^{N_2k}\sum_{\substack{\sum_{i=-N_2}^{N_2}
           j_i=k\\ \sum_{i=-N_2}^{N_2}
           j_i\cdot i=m}}\begin{pmatrix}
         k\\ j_{-N_2}\dots j_{N_2}
     \end{pmatrix}\Pi_{i=-N_2}^{N_2} e^{j_i\lambda_it} p_i^{j_i}  f_{\sum_{i=-N_2}^{N_2} j_i\cdot i}(x)\\
     =&\sum_{k=1}^{N_1}((-\delta)^k \sum_{m=-N_2k}^{N_2k}f_m(x)\sum_{\substack{\sum_{i=-N_2}^{N_2}
           j_i=k\\ \sum_{i=-N_2}^{N_2}
           j_i\cdot i=m}}\begin{pmatrix}
         k\\ j_{-N_2}\dots j_{N_2}
     \end{pmatrix}\Pi_{i=-N_2}^{N_2}  e^{j_i\lambda_it}p_i^{j_i}  \\
          =& \sum_{m=-N_2}^{N_2} f_m(x)\sum_{k=1}^{N_1}((-\delta)^k\sum_{\substack{\sum_{i=-N_2}^{N_2} j_i=k\\ \sum_{i=-N_2}^{N_2} j_i\, i=m}} \begin{pmatrix}
         k\\ j_{-N_2}\dots j_{N_2}
     \end{pmatrix}\Pi_{i=-N_2}^{N_2} e^{j_i\lambda_it} p_i^{j_i} \\
     &\quad + \sum_{j=2}^{N_1} \sum_{m\in [-jN_2,jN_2]\setminus[-(j-1)N_2,(j-1)N_2]}f_m(x)\\
     &\qquad\quad \sum_{k=j}^{N_1} (-\delta)^k \sum_{\substack{\sum_{i=-N_2}^{N_2} j_i=k\\ \sum_{i=-N_2}^{N_2} j_i\, i=m}} \begin{pmatrix}
         k\\ j_{-N_2}\dots j_{N_2}
     \end{pmatrix}\Pi_{i=-N_2}^{N_2}  e^{j_i\lambda_it}p_i^{j_i} 
\end{align*}
One can observe that
\begin{align*}
    &|\sum_{j=2}^{N_1} \sum_{m\in [-jN_2,jN_2]\setminus[-(j-1)N_2,(j-1)N_2]}f_m(x)\sum_{k=j}^{N_1} (-\delta)^k \sum_{\substack{\sum_{i=-N_2}^{N_2} j_i=k\\ \sum_{i=-N_2}^{N_2} j_i\, i=m}} \begin{pmatrix}
         k\\ j_{-N_2}\dots j_{N_2}
     \end{pmatrix}\Pi_{i=-N_2}^{N_2}  e^{j_i\lambda_it}p_i^{j_i}|\\
     \leq &\sum_{j=2}^{N_1} \sum_{m\in [-jN_2,jN_2]\setminus[-(j-1)N_2,(j-1)N_2]}\sum_{k=j}^{N_1} \sum_{\substack{\sum_{i=-N_2}^{N_2} j_i=k\\ \sum_{i=-N_2}^{N_2} j_i\, i=m}}  \begin{pmatrix}
         k\\ j_{-N_2}\dots j_{N_2}
     \end{pmatrix}\Pi_{i=-N_2}^{N_2}  e^{j_i\lambda_it}\gamma^{j_i\cdot i}\\
      = &\sum_{j=2}^{N_1} \sum_{m\in [-jN_2,jN_2]\setminus[-(j-1)N_2,(j-1)N_2]}\gamma^m e^{-\lambda_mt}\sum_{k=j}^{N_1} \sum_{\substack{\sum_{i=-N_2}^{N_2} j_i=k\\ \sum_{i=-N_2}^{N_2} j_i\, i=m}}  \begin{pmatrix}
         k\\ j_{-N_2}\dots j_{N_2} 
     \end{pmatrix}\\
     \leq& 2e^{\lambda_1t}\sum_{j=2}^{N_1}\sum_{m=(j-1)N_2}^{jN_2} \gamma^m \sum_{k=j}^{N_1}\sum_{\substack{\sum_{i=-N_2}^{N_2} j_i=k\\ \sum_{i=-N_2}^{N_2} j_i\, i=m}}  \begin{pmatrix}
         k\\ j_{-N_2}\dots j_{N_2} 
     \end{pmatrix}\\
     \leq &  2e^{\lambda_1t}\sum_{j=2}^{N_1}\sum_{(j-1)N_2}^{jN_2} \gamma^m \sum_{k=j}^{N_1} (2N_2)^k\\
     =&  2e^{\lambda_1t}\sum_{j=2}^{N_1}\sum_{(j-1)N_2}^{jN_2} \gamma^m \frac{(2N_2)^{N_1}-(2N_2)^j}{2N_2-1}\\
     =&2 e^{\lambda_1t}\sum_{j=2}^{N_1}\frac{(2N_2)^{N_1}-(2N_2)^j}{2N_2-1}\frac{\gamma^{jN_2+1}-\gamma^{(j-1)N_2}}{\gamma-1}\\
     =&e^{\lambda_1t}\frac{2(1+N_1)}{(2N_2-1)} \left[ (2N_2)^{N_1}\gamma^{N_2}\frac{ \gamma^{N_2+1}-1}{\gamma-1}\frac{\gamma^{(N_1-2)N_2+1}-1}{\gamma^{N_2}-1}\right]\\
     \leq& e^{\lambda_1t}\frac{2(1+N_1)}{(2N_2-1)} \left[ (2N_2)^{N_1}\gamma^{N_2}\right]
\end{align*}

Now choose $N_2$ such that 
$$2N_1 \frac{\delta^2(1-\delta^{N_1})}{1-\delta} \frac{\gamma^{N_2}}{1-\gamma}\leq \frac{\varepsilon}4$$  as well as $$\frac{2(1+N_1)}{(2N_2-1)} \left[ (2N_2)^{N_1}\gamma^{N_2}\right]\leq\frac{\varepsilon}4.$$
Define $\Pi_{N_2} = 1+\sum_{m=-N_2}^{N_2} f_m(x)\sum_{k=1}^{N_1}((-\delta)^k\sum_{\substack{\sum_{i=-N_2}^{N_2} j_i=k\\ \sum_{i=-N_2}^{N_2} j_i\, i=m}} \begin{pmatrix}
         k\\ j_{-N_2}\dots j_{N_2}
     \end{pmatrix}\Pi_{i=-N_2}^{N_2} e^{j_i\lambda_it} p_i^{j_i}$
Then
\begin{align*}
    \left|\frac{1}{\displaystyle1+\delta \sum_{i=-\infty}^\infty p_i e^{\lambda_i t}f_i(x)}\,-  \Pi_{N_2}\right|\leq \varepsilon e^{\lambda_1t}
\end{align*}
\end{proof}

\end{appendices}

\bibliography{sn-bibliography}


\begin{thebibliography}{34}
\ifx \bisbn   \undefined \def \bisbn  #1{ISBN #1}\fi
\ifx \binits  \undefined \def \binits#1{#1}\fi
\ifx \bauthor  \undefined \def \bauthor#1{#1}\fi
\ifx \batitle  \undefined \def \batitle#1{#1}\fi
\ifx \bjtitle  \undefined \def \bjtitle#1{#1}\fi
\ifx \bvolume  \undefined \def \bvolume#1{\textbf{#1}}\fi
\ifx \byear  \undefined \def \byear#1{#1}\fi
\ifx \bissue  \undefined \def \bissue#1{#1}\fi
\ifx \bfpage  \undefined \def \bfpage#1{#1}\fi
\ifx \blpage  \undefined \def \blpage #1{#1}\fi
\ifx \burl  \undefined \def \burl#1{\textsf{#1}}\fi
\ifx \doiurl  \undefined \def \doiurl#1{\url{https://doi.org/#1}}\fi
\ifx \betal  \undefined \def \betal{\textit{et al.}}\fi
\ifx \binstitute  \undefined \def \binstitute#1{#1}\fi
\ifx \binstitutionaled  \undefined \def \binstitutionaled#1{#1}\fi
\ifx \bctitle  \undefined \def \bctitle#1{#1}\fi
\ifx \beditor  \undefined \def \beditor#1{#1}\fi
\ifx \bpublisher  \undefined \def \bpublisher#1{#1}\fi
\ifx \bbtitle  \undefined \def \bbtitle#1{#1}\fi
\ifx \bedition  \undefined \def \bedition#1{#1}\fi
\ifx \bseriesno  \undefined \def \bseriesno#1{#1}\fi
\ifx \blocation  \undefined \def \blocation#1{#1}\fi
\ifx \bsertitle  \undefined \def \bsertitle#1{#1}\fi
\ifx \bsnm \undefined \def \bsnm#1{#1}\fi
\ifx \bsuffix \undefined \def \bsuffix#1{#1}\fi
\ifx \bparticle \undefined \def \bparticle#1{#1}\fi
\ifx \barticle \undefined \def \barticle#1{#1}\fi
\bibcommenthead
\ifx \bconfdate \undefined \def \bconfdate #1{#1}\fi
\ifx \botherref \undefined \def \botherref #1{#1}\fi
\ifx \url \undefined \def \url#1{\textsf{#1}}\fi
\ifx \bchapter \undefined \def \bchapter#1{#1}\fi
\ifx \bbook \undefined \def \bbook#1{#1}\fi
\ifx \bcomment \undefined \def \bcomment#1{#1}\fi
\ifx \oauthor \undefined \def \oauthor#1{#1}\fi
\ifx \citeauthoryear \undefined \def \citeauthoryear#1{#1}\fi
\ifx \endbibitem  \undefined \def \endbibitem {}\fi
\ifx \bconflocation  \undefined \def \bconflocation#1{#1}\fi
\ifx \arxivurl  \undefined \def \arxivurl#1{\textsf{#1}}\fi
\csname PreBibitemsHook\endcsname

\bibitem[\protect\citeauthoryear{Hyv{\"a}rinen and Dayan}{2005}]{hyvarinen2005estimation}
\begin{botherref}
\oauthor{\bsnm{Hyv{\"a}rinen}, \binits{A.}},
\oauthor{\bsnm{Dayan}, \binits{P.}}:
Estimation of non-normalized statistical models by score matching.
Journal of Machine Learning Research
\textbf{6}(4)
(2005)
\end{botherref}
\endbibitem

\bibitem[\protect\citeauthoryear{Anderson}{1982}]{anderson1982reverse}
\begin{barticle}
\bauthor{\bsnm{Anderson}, \binits{B.D.O.}}:
\batitle{Reverse-time diffusion equation models}.
\bjtitle{Stochastic Processes and their Applications}
\bvolume{12}(\bissue{3}),
\bfpage{313}--\blpage{326}
(\byear{1982})
\end{barticle}
\endbibitem

\bibitem[\protect\citeauthoryear{Song and Ermon}{2019}]{song2019generative}
\begin{botherref}
\oauthor{\bsnm{Song}, \binits{Y.}},
\oauthor{\bsnm{Ermon}, \binits{S.}}:
Generative modeling by estimating gradients of the data distribution.
Advances in neural information processing systems
\textbf{32}
(2019)
\end{botherref}
\endbibitem

\bibitem[\protect\citeauthoryear{Dathathri et~al.}{2019}]{dathathri2019plug}
\begin{botherref}
\oauthor{\bsnm{Dathathri}, \binits{S.}},
\oauthor{\bsnm{Madotto}, \binits{A.}},
\oauthor{\bsnm{Lan}, \binits{J.}},
\oauthor{\bsnm{Hung}, \binits{J.}},
\oauthor{\bsnm{Frank}, \binits{E.}},
\oauthor{\bsnm{Molino}, \binits{P.}},
\oauthor{\bsnm{Yosinski}, \binits{J.}},
\oauthor{\bsnm{Liu}, \binits{R.}}:
Plug and play language models: A simple approach to controlled text generation.
arXiv preprint arXiv:1912.02164
(2019)
\end{botherref}
\endbibitem

\bibitem[\protect\citeauthoryear{Song and Ermon}{2020}]{song2020improved}
\begin{barticle}
\bauthor{\bsnm{Song}, \binits{Y.}},
\bauthor{\bsnm{Ermon}, \binits{S.}}:
\batitle{Improved techniques for training score-based generative models}.
\bjtitle{Advances in neural information processing systems}
\bvolume{33},
\bfpage{12438}--\blpage{12448}
(\byear{2020})
\end{barticle}
\endbibitem

\bibitem[\protect\citeauthoryear{Song et~al.}{2020}]{song2020score}
\begin{botherref}
\oauthor{\bsnm{Song}, \binits{Y.}},
\oauthor{\bsnm{Sohl-Dickstein}, \binits{J.}},
\oauthor{\bsnm{Kingma}, \binits{D.P.}},
\oauthor{\bsnm{Kumar}, \binits{A.}},
\oauthor{\bsnm{Ermon}, \binits{S.}},
\oauthor{\bsnm{Poole}, \binits{B.}}:
Score-based generative modeling through stochastic differential equations.
arXiv preprint arXiv:2011.13456
(2020)
\end{botherref}
\endbibitem

\bibitem[\protect\citeauthoryear{Meng et~al.}{2021}]{meng2021sdedit}
\begin{botherref}
\oauthor{\bsnm{Meng}, \binits{C.}},
\oauthor{\bsnm{He}, \binits{Y.}},
\oauthor{\bsnm{Song}, \binits{Y.}},
\oauthor{\bsnm{Song}, \binits{J.}},
\oauthor{\bsnm{Wu}, \binits{J.}},
\oauthor{\bsnm{Zhu}, \binits{J.-Y.}},
\oauthor{\bsnm{Ermon}, \binits{S.}}:
Sdedit: Guided image synthesis and editing with stochastic differential equations.
arXiv preprint arXiv:2108.01073
(2021)
\end{botherref}
\endbibitem

\bibitem[\protect\citeauthoryear{Lee et~al.}{2022}]{lee2022convergence}
\begin{barticle}
\bauthor{\bsnm{Lee}, \binits{H.}},
\bauthor{\bsnm{Lu}, \binits{J.}},
\bauthor{\bsnm{Tan}, \binits{Y.}}:
\batitle{Convergence for score-based generative modeling with polynomial complexity}.
\bjtitle{Advances in Neural Information Processing Systems}
\bvolume{35},
\bfpage{22870}--\blpage{22882}
(\byear{2022})
\end{barticle}
\endbibitem

\bibitem[\protect\citeauthoryear{Lee et~al.}{2023}]{lee2023convergence}
\begin{bchapter}
\bauthor{\bsnm{Lee}, \binits{H.}},
\bauthor{\bsnm{Lu}, \binits{J.}},
\bauthor{\bsnm{Tan}, \binits{Y.}}:
\bctitle{Convergence of score-based generative modeling for general data distributions}.
In: \bbtitle{International Conference on Algorithmic Learning Theory},
pp. \bfpage{946}--\blpage{985}
(\byear{2023}).
\bcomment{PMLR}
\end{bchapter}
\endbibitem

\bibitem[\protect\citeauthoryear{Ho et~al.}{2020}]{ho2020denoising}
\begin{barticle}
\bauthor{\bsnm{Ho}, \binits{J.}},
\bauthor{\bsnm{Jain}, \binits{A.}},
\bauthor{\bsnm{Abbeel}, \binits{P.}}:
\batitle{Denoising diffusion probabilistic models}.
\bjtitle{Advances in neural information processing systems}
\bvolume{33},
\bfpage{6840}--\blpage{6851}
(\byear{2020})
\end{barticle}
\endbibitem

\bibitem[\protect\citeauthoryear{Sohl-Dickstein et~al.}{2015}]{sohl2015deep}
\begin{bchapter}
\bauthor{\bsnm{Sohl-Dickstein}, \binits{J.}},
\bauthor{\bsnm{Weiss}, \binits{E.}},
\bauthor{\bsnm{Maheswaranathan}, \binits{N.}},
\bauthor{\bsnm{Ganguli}, \binits{S.}}:
\bctitle{Deep unsupervised learning using nonequilibrium thermodynamics}.
In: \bbtitle{International Conference on Machine Learning},
pp. \bfpage{2256}--\blpage{2265}
(\byear{2015}).
\bcomment{pmlr}
\end{bchapter}
\endbibitem

\bibitem[\protect\citeauthoryear{Song et~al.}{2021}]{song2021maximum}
\begin{barticle}
\bauthor{\bsnm{Song}, \binits{Y.}},
\bauthor{\bsnm{Durkan}, \binits{C.}},
\bauthor{\bsnm{Murray}, \binits{I.}},
\bauthor{\bsnm{Ermon}, \binits{S.}}:
\batitle{Maximum likelihood training of score-based diffusion models}.
\bjtitle{Advances in neural information processing systems}
\bvolume{34},
\bfpage{1415}--\blpage{1428}
(\byear{2021})
\end{barticle}
\endbibitem

\bibitem[\protect\citeauthoryear{Vahdat et~al.}{2021}]{vahdat2021score}
\begin{barticle}
\bauthor{\bsnm{Vahdat}, \binits{A.}},
\bauthor{\bsnm{Kreis}, \binits{K.}},
\bauthor{\bsnm{Kautz}, \binits{J.}}:
\batitle{Score-based generative modeling in latent space}.
\bjtitle{Advances in neural information processing systems}
\bvolume{34},
\bfpage{11287}--\blpage{11302}
(\byear{2021})
\end{barticle}
\endbibitem

\bibitem[\protect\citeauthoryear{Lipman et~al.}{2022}]{lipman2022flow}
\begin{botherref}
\oauthor{\bsnm{Lipman}, \binits{Y.}},
\oauthor{\bsnm{Chen}, \binits{R.T.}},
\oauthor{\bsnm{Ben-Hamu}, \binits{H.}},
\oauthor{\bsnm{Nickel}, \binits{M.}},
\oauthor{\bsnm{Le}, \binits{M.}}:
Flow matching for generative modeling.
arXiv preprint arXiv:2210.02747
(2022)
\end{botherref}
\endbibitem

\bibitem[\protect\citeauthoryear{Liu et~al.}{2022}]{liu2022flow}
\begin{botherref}
\oauthor{\bsnm{Liu}, \binits{X.}},
\oauthor{\bsnm{Gong}, \binits{C.}},
\oauthor{\bsnm{Liu}, \binits{Q.}}:
Flow straight and fast: Learning to generate and transfer data with rectified flow.
arXiv preprint arXiv:2209.03003
(2022)
\end{botherref}
\endbibitem

\bibitem[\protect\citeauthoryear{Liu}{2022}]{liu2022rectified}
\begin{botherref}
\oauthor{\bsnm{Liu}, \binits{Q.}}:
Rectified flow: A marginal preserving approach to optimal transport.
arXiv preprint arXiv:2209.14577
(2022)
\end{botherref}
\endbibitem

\bibitem[\protect\citeauthoryear{Albergo and Vanden-Eijnden}{2022}]{albergo2022building}
\begin{botherref}
\oauthor{\bsnm{Albergo}, \binits{M.S.}},
\oauthor{\bsnm{Vanden-Eijnden}, \binits{E.}}:
Building normalizing flows with stochastic interpolants.
arXiv preprint arXiv:2209.15571
(2022)
\end{botherref}
\endbibitem

\bibitem[\protect\citeauthoryear{Albergo et~al.}{2023a}]{albergo2023stochastic}
\begin{botherref}
\oauthor{\bsnm{Albergo}, \binits{M.S.}},
\oauthor{\bsnm{Boffi}, \binits{N.M.}},
\oauthor{\bsnm{Vanden-Eijnden}, \binits{E.}}:
Stochastic interpolants: A unifying framework for flows and diffusions.
arXiv preprint arXiv:2303.08797
(2023)
\end{botherref}
\endbibitem

\bibitem[\protect\citeauthoryear{Albergo et~al.}{2023b}]{albergo2023stochastic2}
\begin{botherref}
\oauthor{\bsnm{Albergo}, \binits{M.S.}},
\oauthor{\bsnm{Goldstein}, \binits{M.}},
\oauthor{\bsnm{Boffi}, \binits{N.M.}},
\oauthor{\bsnm{Ranganath}, \binits{R.}},
\oauthor{\bsnm{Vanden-Eijnden}, \binits{E.}}:
Stochastic interpolants with data-dependent couplings.
arXiv preprint arXiv:2310.03725
(2023)
\end{botherref}
\endbibitem

\bibitem[\protect\citeauthoryear{Yang et~al.}{2023}]{yang2023diffusion}
\begin{barticle}
\bauthor{\bsnm{Yang}, \binits{L.}},
\bauthor{\bsnm{Zhang}, \binits{Z.}},
\bauthor{\bsnm{Song}, \binits{Y.}},
\bauthor{\bsnm{Hong}, \binits{S.}},
\bauthor{\bsnm{Xu}, \binits{R.}},
\bauthor{\bsnm{Zhao}, \binits{Y.}},
\bauthor{\bsnm{Zhang}, \binits{W.}},
\bauthor{\bsnm{Cui}, \binits{B.}},
\bauthor{\bsnm{Yang}, \binits{M.-H.}}:
\batitle{Diffusion models: A comprehensive survey of methods and applications}.
\bjtitle{ACM Computing Surveys}
\bvolume{56}(\bissue{4}),
\bfpage{1}--\blpage{39}
(\byear{2023})
\end{barticle}
\endbibitem

\bibitem[\protect\citeauthoryear{Chan et~al.}{2024}]{chan2024tutorial}
\begin{barticle}
\bauthor{\bsnm{Chan}, \binits{S.}}, \betal:
\batitle{Tutorial on diffusion models for imaging and vision}.
\bjtitle{Foundations and Trends{\textregistered} in Computer Graphics and Vision}
\bvolume{16}(\bissue{4}),
\bfpage{322}--\blpage{471}
(\byear{2024})
\end{barticle}
\endbibitem

\bibitem[\protect\citeauthoryear{Ma et~al.}{2025}]{ma2025efficient}
\begin{botherref}
\oauthor{\bsnm{Ma}, \binits{Z.}},
\oauthor{\bsnm{Zhang}, \binits{Y.}},
\oauthor{\bsnm{Jia}, \binits{G.}},
\oauthor{\bsnm{Zhao}, \binits{L.}},
\oauthor{\bsnm{Ma}, \binits{Y.}},
\oauthor{\bsnm{Ma}, \binits{M.}},
\oauthor{\bsnm{Liu}, \binits{G.}},
\oauthor{\bsnm{Zhang}, \binits{K.}},
\oauthor{\bsnm{Ding}, \binits{N.}},
\oauthor{\bsnm{Li}, \binits{J.}}, et al.:
Efficient diffusion models: A comprehensive survey from principles to practices.
IEEE Transactions on Pattern Analysis and Machine Intelligence
(2025)
\end{botherref}
\endbibitem

\bibitem[\protect\citeauthoryear{Chen et~al.}{2023}]{chen2023score}
\begin{bchapter}
\bauthor{\bsnm{Chen}, \binits{M.}},
\bauthor{\bsnm{Huang}, \binits{K.}},
\bauthor{\bsnm{Zhao}, \binits{T.}},
\bauthor{\bsnm{Wang}, \binits{M.}}:
\bctitle{Score approximation, estimation and distribution recovery of diffusion models on low-dimensional data}.
In: \bbtitle{International Conference on Machine Learning},
pp. \bfpage{4672}--\blpage{4712}
(\byear{2023}).
\bcomment{PMLR}
\end{bchapter}
\endbibitem

\bibitem[\protect\citeauthoryear{Oko et~al.}{2023}]{oko2023diffusion}
\begin{bchapter}
\bauthor{\bsnm{Oko}, \binits{K.}},
\bauthor{\bsnm{Akiyama}, \binits{S.}},
\bauthor{\bsnm{Suzuki}, \binits{T.}}:
\bctitle{Diffusion models are minimax optimal distribution estimators}.
In: \bbtitle{International Conference on Machine Learning},
pp. \bfpage{26517}--\blpage{26582}
(\byear{2023}).
\bcomment{PMLR}
\end{bchapter}
\endbibitem

\bibitem[\protect\citeauthoryear{Block et~al.}{2020}]{block2020generative}
\begin{botherref}
\oauthor{\bsnm{Block}, \binits{A.}},
\oauthor{\bsnm{Mroueh}, \binits{Y.}},
\oauthor{\bsnm{Rakhlin}, \binits{A.}}:
Generative modeling with denoising auto-encoders and langevin sampling.
arXiv preprint arXiv:2002.00107
(2020)
\end{botherref}
\endbibitem

\bibitem[\protect\citeauthoryear{Han et~al.}{2024}]{han2024neural}
\begin{botherref}
\oauthor{\bsnm{Han}, \binits{Y.}},
\oauthor{\bsnm{Razaviyayn}, \binits{M.}},
\oauthor{\bsnm{Xu}, \binits{R.}}:
Neural network-based score estimation in diffusion models: Optimization and generalization.
arXiv preprint arXiv:2401.15604
(2024)
\end{botherref}
\endbibitem

\bibitem[\protect\citeauthoryear{Chen et~al.}{2024}]{chen2024overview}
\begin{botherref}
\oauthor{\bsnm{Chen}, \binits{M.}},
\oauthor{\bsnm{Mei}, \binits{S.}},
\oauthor{\bsnm{Fan}, \binits{J.}},
\oauthor{\bsnm{Wang}, \binits{M.}}:
An overview of diffusion models: Applications, guided generation, statistical rates and optimization.
arXiv preprint arXiv:2404.07771
(2024)
\end{botherref}
\endbibitem

\bibitem[\protect\citeauthoryear{Pavliotis}{2014}]{pavliotis2014stochastic}
\begin{botherref}
\oauthor{\bsnm{Pavliotis}, \binits{G.A.}}:
Stochastic processes and applications.
Texts in applied mathematics
\textbf{60}
(2014)
\end{botherref}
\endbibitem

\bibitem[\protect\citeauthoryear{Boyd and Vandenberghe}{2004}]{boyd2004convex}
\begin{bbook}
\bauthor{\bsnm{Boyd}, \binits{S.}},
\bauthor{\bsnm{Vandenberghe}, \binits{L.}}:
\bbtitle{Convex Optimization}.
\bpublisher{Cambridge university press}, \blocation{???}
(\byear{2004})
\end{bbook}
\endbibitem

\bibitem[\protect\citeauthoryear{Wainwright et~al.}{2008}]{wainwright2008graphical}
\begin{barticle}
\bauthor{\bsnm{Wainwright}, \binits{M.J.}},
\bauthor{\bsnm{Jordan}, \binits{M.I.}}, \betal:
\batitle{Graphical models, exponential families, and variational inference}.
\bjtitle{Foundations and Trends{\textregistered} in Machine Learning}
\bvolume{1}(\bissue{1--2}),
\bfpage{1}--\blpage{305}
(\byear{2008})
\end{barticle}
\endbibitem

\bibitem[\protect\citeauthoryear{Chen and Khoo}{2023}]{chen2023combining}
\begin{botherref}
\oauthor{\bsnm{Chen}, \binits{Y.}},
\oauthor{\bsnm{Khoo}, \binits{Y.}}:
Combining particle and tensor-network methods for partial differential equations via sketching.
arXiv preprint arXiv:2305.17884
(2023)
\end{botherref}
\endbibitem

\bibitem[\protect\citeauthoryear{Peng et~al.}{2023}]{peng2023generative}
\begin{botherref}
\oauthor{\bsnm{Peng}, \binits{Y.}},
\oauthor{\bsnm{Chen}, \binits{Y.}},
\oauthor{\bsnm{Stoudenmire}, \binits{E.M.}},
\oauthor{\bsnm{Khoo}, \binits{Y.}}:
Generative modeling via hierarchical tensor sketching.
arXiv preprint arXiv:2304.05305
(2023)
\end{botherref}
\endbibitem

\bibitem[\protect\citeauthoryear{Liberty et~al.}{2007}]{liberty2007randomized}
\begin{barticle}
\bauthor{\bsnm{Liberty}, \binits{E.}},
\bauthor{\bsnm{Woolfe}, \binits{F.}},
\bauthor{\bsnm{Martinsson}, \binits{P.-G.}},
\bauthor{\bsnm{Rokhlin}, \binits{V.}},
\bauthor{\bsnm{Tygert}, \binits{M.}}:
\batitle{Randomized algorithms for the low-rank approximation of matrices}.
\bjtitle{Proceedings of the National Academy of Sciences}
\bvolume{104}(\bissue{51}),
\bfpage{20167}--\blpage{20172}
(\byear{2007})
\end{barticle}
\endbibitem

\bibitem[\protect\citeauthoryear{Halko et~al.}{2011}]{halko2011finding}
\begin{barticle}
\bauthor{\bsnm{Halko}, \binits{N.}},
\bauthor{\bsnm{Martinsson}, \binits{P.-G.}},
\bauthor{\bsnm{Tropp}, \binits{J.A.}}:
\batitle{Finding structure with randomness: Probabilistic algorithms for constructing approximate matrix decompositions}.
\bjtitle{SIAM review}
\bvolume{53}(\bissue{2}),
\bfpage{217}--\blpage{288}
(\byear{2011})
\end{barticle}
\endbibitem

\end{thebibliography}

\end{document}